\documentclass[12pt]{article}
\usepackage{amsmath,amssymb,amsthm,fullpage}

\newtheorem{thm}{Theorem}[section]
\newtheorem{lem}[thm]{Lemma}
\title{Exact local spectral thresholds for perfect matchings\\ in $3$-graphs and $3$-partite $3$-graphs}
\author{Pei Liu\thanks{Department of Mathematics, Sungkyunkwan University, Suwon, 16419, Republic of Korea. liupei2023@g.skku.edu. Research supported by the China Scholarship Council.}
	\ \ and\ \ Suil O\thanks{Department of Applied Mathematics and Statistics, The State University of New York, Korea, Incheon, 21985, suil.o@sunykorea.ac.kr. Corresponding author. Research supported by the National Research Foundation of Korea (NRF) grant funded by the Korea government(MSIT) No. RS-2025-23523950.}}
\date{}

\begin{document}
	\maketitle
	
	\begin{abstract}
		For a $3$-uniform hypergraph $H$, let $\sigma(H)$ be the minimum over the vertices
		of $H$ of the spectral radius of the link. Lin, Lu, Yuan and Zhao conjectured that
		$\sigma(H)>\tfrac{2n}3-2$ forces a perfect matching in a $3$-graph of order $n$
		divisible by three, and Lu and Yuan conjectured that $\sigma(H)>\tau(q)$ forces one
		in a $q$-balanced $3$-partite $3$-graph, where $\tau(q)=\sqrt{q(q-1)/2}$ for odd $q$
		and $\tau(q)$ is given by a quartic for even $q$. In this paper, we prove both perfect matching conjectures for large order. Each proof passes through a stability theorem for
		fractional matchings: a fractional vertex cover of deficient weight with a
		zero coordinate is obtained, the link of that vertex carries the induced cover, and the
		spectral radius of such a link is bounded by an inequality whose extremal cases identify
		the corresponding space barriers.
	\end{abstract}
	
	\noindent\emph{Keywords}: perfect matching; fractional matching; 3-graph; 3-partite 3-graph; spectral radius; link graph.
	
	\medskip
	\noindent\emph{AMS 2020 subject classification}: 05C50, 05C65, 05C70, 05D05.
	
	\section{Introduction}\label{sec:intro}
	
	A \emph{$3$-uniform hypergraph}, or \emph{$3$-graph}, is a pair $H=(V(H),E(H))$
	with $E(H)\subseteq\binom{V(H)}3$. A \emph{matching} is a set of pairwise disjoint
	edges, and it is \emph{perfect} if its edges cover $V(H)$. A $3$-graph is
	\emph{$q$-balanced $3$-partite} if $V(H)$ is partitioned into classes
	$V_1,V_2,V_3$ of size $q$ meeting every edge once each; a perfect matching then has
	$q$ edges. For $v\in V(H)$ the \emph{link} $L_H(v)$ is the graph on
	$V(H)\setminus\{v\}$ with edge set $\{e\setminus\{v\}:v\in e\in E(H)\}$, which is
	bipartite with parts $V_j$ and $V_k$ when $H$ is $3$-partite and $v\in V_i$. Write
	$\rho$ for the adjacency spectral radius and
	\[
	\sigma(H)=\min_{v\in V(H)}\rho\bigl(L_H(v)\bigr).
	\]
	
	Local spectral conditions provide a natural counterpart to
	vertex-degree conditions for perfect matchings.
	The vertex degree $d_H(v)$ equals the number of edges in
	the link $L_H(v)$, while its spectral radius depends on
	both the number and the arrangement of these edges.
	This motivates the determination of exact thresholds
	on $\sigma(H)$ that guarantee a perfect matching. Lin, Lu, Yuan and Zhao~\cite{LLYZ} proved that
	$\sigma(H)>(\tfrac23+\gamma)n$ forces a perfect matching when $3\mid n$ and $n$ is
	large, and conjectured a sharp condition whose perfect matching endpoint reads
	$\sigma(H)>\tfrac{2n}3-2$; they also proved a fractional matching version of their
	conjecture~\cite[Theorem~1.4]{LLYZ}. Lu and Yuan~\cite{LY} proved the partite analogue with
	$(\tfrac{\sqrt2}2+\gamma)q$ and conjectured the exact threshold
	\begin{equation}\label{eq:tau}
		\tau(q)=
		\begin{cases}
			\sqrt{\dfrac{q(q-1)}2}, & q\ \text{odd},\\[2ex]
			\dfrac12\sqrt{q^2-q+2+\sqrt{q^4-6q^3+9q^2+12q-12}}, & q\ \text{even}.
		\end{cases}
	\end{equation}
	In this paper, we prove both perfect matching conjectures for large order.
	
	\begin{thm}\label{thm:nonpartite}
		There is $n_0$ such that every $3$-graph $H$ on $n\ge n_0$ vertices with $3\mid n$
		and $\sigma(H)>\tfrac{2n}3-2$ has a perfect matching.
	\end{thm}
	
	\begin{thm}\label{thm:partite}
		There is $q_0$ such that every $q$-balanced $3$-partite $3$-graph $H$ with
		$q\ge q_0$ and $\sigma(H)>\tau(q)$ has a perfect matching.
	\end{thm}
	
	Both thresholds are attained by space barriers. In the non-partite setting, fix $A\subseteq V$ with
	$|A|=\tfrac{in}3-1$ for $i\in\{1,2\}$ and let $T_i(A)$ consist of the triples
	meeting $A$ in at least $i$ vertices; a matching in $T_i(A)$ has fewer than $n/3$
	edges, and $\sigma(T_i(A))=\tfrac{2n}3-2$. In the partite setting, for $q\ge3$, fix
	$C_i\subseteq V_i$ with $|C_1|+|C_2|+|C_3|=q-1$ and take every legal triple meeting
	$C_1\cup C_2\cup C_3$; the choices
	\[
	\bigl(|C_1|,|C_2|,|C_3|\bigr)=
	\begin{cases}
		\bigl(\tfrac{q-1}2,\tfrac{q-1}2,0\bigr), & q\ \text{odd},\\[1ex]
		\bigl(\tfrac{q-2}2,\tfrac{q-2}2,1\bigr), & q\ \text{even}
	\end{cases}
	\]
	give $\sigma=\tau(q)$. The quartic in \eqref{eq:tau} is the discriminant of a
	quadratic attached to the even construction, whose critical links carry one block
	of size one; this is why the two parities cannot be treated by a single formula.
	
	The two proofs share an architecture. Absence of a perfect fractional matching
	gives, by linear programming duality, a fractional vertex cover of deficient
	weight; obtaining a zero-weight vertex and passing to its link produces a graph carrying an induced cover of restricted weight, and the
	spectral radius of such a graph obeys an inequality whose equality cases are
	exactly the barriers. Quantifying that inequality yields a stability theorem, and
	the remaining steps are an absorbing set built below the critical value, an almost
	perfect matching through the nibble, and a direct analysis near the barriers.
	
	The partite problem is the more rigid of the two. Its barriers form a single
	family, and the two barrier sets have sizes close to $q/2$; the links are bipartite, and the asymmetric cover constraints
	are handled by Lemma~\ref{lem:bipineq}; and a matching in a
	$3$-partite $3$-graph leaves a balanced set uncovered for free. Only in the
	extremal analysis is the partite case the harder of the two: repairing the barrier can change its size in either direction,
	so a separate matching construction is needed to correct
	the resulting imbalance.
	
	For notation not defined here we follow West~\cite{West}; for spectral background
	we follow Brouwer and Haemers~\cite{BH} and Godsil and Royle~\cite{GR}. The paper is organized as follows.
	Section~\ref{sec:tools} collects definitions and auxiliary results.
	Section~\ref{sec:fsnon} establishes fractional matching stability
	in both settings.
	Section~\ref{sec:pm} proves Theorems~\ref{thm:nonpartite}
	and~\ref{thm:partite} using absorption and separate analyses
	of the nonextremal and extremal cases.
	Section~\ref{sec:conc} discusses further directions.
	
	\section{Definitions and tools}\label{sec:tools}
	
	A \emph{fractional matching} in a hypergraph $F$ is a function $f:E(F)\to[0,1]$
	with $\sum_{e\ni v}f(e)\le1$ for every vertex; its \emph{size} is $\sum_ef(e)$, the
	maximum size is $\nu^*(F)$, and $f$ is \emph{perfect} if every vertex constraint is
	tight. A \emph{fractional vertex cover} is a function $g:V(F)\to[0,1]$ with
	$\sum_{v\in e}g(v)\ge1$ for every edge; the minimum of $g(V(F))$ is $\tau^*(F)$. We
	write $g(U)=\sum_{v\in U}g(v)$, $\Delta_1$ for the maximum vertex degree of a
	hypergraph, and $e(G)$ for the number of edges of a graph $G$. For a bipartite
	graph with specified classes, the \emph{biadjacency matrix} has largest singular
	value $\rho(G)$.
	
	\begin{lem}\label{lem:basic}
		Let $F$ be a hypergraph and let $G$, $G'$ be graphs.
		\begin{enumerate}
			\item[\rm(i)] $\nu^*(F)=\tau^*(F)$; and if $f$ is a maximum fractional matching and
			$g$ a minimum fractional vertex cover, then $g(v)>0$ forces $\sum_{e\ni v}f(e)=1$
			and $f(e)>0$ forces $\sum_{v\in e}g(v)=1$.
			\item[\rm(ii)] If $G$ is bipartite, then $\rho(G)^2\le e(G)$.
			Equality holds if and only if $G$ is edgeless or consists
			of a complete bipartite component together with isolated
			vertices. In general,
			\(
			\rho(G)^2+\rho(G)\le2e(G).
			\)
			\item[\rm(iii)] If $G'$ arises from $G$ by deleting $k$ edges, then
			$\rho(G')\ge\rho(G)-\sqrt{2k}$, and $\rho(G')\ge\rho(G)-\sqrt k$ when $G$ is
			bipartite.
			\item[\rm(iv)] If $G$ has at most $N$ vertices, $\rho(G)>0$ and $R\subseteq V(G)$,
			then $\rho(G-R)\ge\rho(G)-2|R|N/\rho(G)$; in particular
			$\rho(G-R)\ge\rho(G)-4|R|$ if $\rho(G)\ge N/2$, and $\rho(G-R)\ge\rho(G)-8|R|$ if
			$\rho(G)\ge N/4$.
		\end{enumerate}
	\end{lem}
	
	\begin{proof}
		Part (i) is linear programming duality together with complementary slackness. The
		first half of (ii) is the singular value bound $\rho(G)^2\le\|B\|_F^2=e(G)$ for the
		biadjacency matrix $B$, with equality precisely when $B$ has rank at most one, and the second
		half is the bound of Stanley~\cite{Stanley}. Part (iii) follows from
		$\rho(G)\le\rho(G')+\rho(G-E(G'))$ together with the two bounds in (ii) applied to
		$G-E(G')$, a subgraph of $G$ with $k$ edges.
		
		For (iv), let $\mathbf p$ be a nonnegative unit Perron vector of $G$.
		Define $\mathbf p'$ on $V(G)$ by $p'_i=p_i$ for $i\notin R$
		and $p'_i=0$ for $i\in R$, and put $\lambda=\rho(G)$
		and $\theta=\sum_{i\in R}p_i^2$. For each $i$, $\lambda p_i=\sum_{j\sim i}p_j\le\sqrt N$
		by Cauchy--Schwarz, so $p_i^2\le N/\lambda^2$ and $\theta\le|R|N/\lambda^2$. If
		$\theta>\tfrac12$ then $2|R|N/\lambda\ge2\lambda\theta>\lambda$ and the assertion is
		vacuous, so assume $\theta\le\tfrac12$. Then
		$\mathbf p'^{\top}A(G)\mathbf p'\ge\lambda(1-2\theta)$ and
		$\|\mathbf p'\|^2=1-\theta\le1$, so
		\[
		\rho(G-R)\ \ge\ \lambda\frac{1-2\theta}{1-\theta}\ \ge\ \lambda(1-2\theta)
		\ \ge\ \lambda-\frac{2|R|N}\lambda ,
		\]
		and the two special cases follow from $N/\lambda\le2$ and $N/\lambda\le4$.
	\end{proof}
	
	The next estimate converts a spectral deficit into a bound on the number of absent
	edges. It is used in both settings, for two families of host graphs.
	
	\begin{lem}\label{lem:missing}
		Let $T$ be a graph on $N$ vertices with $\lambda=\rho(T)>0$ all of whose other
		eigenvalues are nonpositive, and let $\mathbf u$ be a positive unit Perron vector
		of $T$ with $u_i^2\ge\beta>0$ for every $i$. If $K$ is a spanning subgraph of $T$ and
		$D=\lambda-\rho(K)$, then
		\[
		\bigl|E(T)\setminus E(K)\bigr|\ \le\ \frac{D}{\beta}\Bigl(\frac{8N}{\lambda}+2\Bigr).
		\]
	\end{lem}
	
	\begin{proof}
		Let $\mathbf z$ be a nonnegative unit Perron vector of $K$ and
		$a=\langle\mathbf z,\mathbf u\rangle$. Since $A(T)-A(K)$ is nonnegative and every
		eigenvalue of $T$ other than $\lambda$ is nonpositive,
		$\lambda-D=\mathbf z^{\top}A(K)\mathbf z\le\mathbf z^{\top}A(T)\mathbf z\le\lambda a^2$,
		so $\|\mathbf z-\mathbf u\|^2=2(1-a)\le2D/\lambda$. Put $B=\{i:z_i<u_i/2\}$. For
		$i\in B$ one has $(z_i-u_i)^2>u_i^2/4\ge\beta/4$, whence $|B|\le8D/(\lambda\beta)$,
		and at most $N|B|$ absent edges meet $B$. For an absent edge $ij$ with
		$i,j\notin B$ we have $2z_iz_j\ge\beta/2$, while
		$\sum_{ij\in E(T)\setminus E(K)}2z_iz_j=\mathbf z^{\top}(A(T)-A(K))\mathbf z\le D$,
		so there are at most $2D/\beta$ of them.
	\end{proof}
	
	Both settings need the spectral radius of a graph built from two blocks. For
	$0\le a,b\le q$ let $J_q(a,b)$ be the bipartite graph with classes of size $q$
	whose edges are the pairs meeting a fixed $a$-set in the first class or a fixed
	$b$-set in the second, and for $0\le p\le N$ let $S_{N,p}=K_p\vee I_{N-p}$.
	
	\begin{lem}\label{lem:blocks}
		Put $T=q(a+b)-ab$ and $D=ab(q-a)(q-b)$.
		\begin{enumerate}
			\item[\rm(i)] $\rho(J_q(a,b))^2=\tfrac12\bigl(T+\sqrt{T^2-4D}\bigr)$; in particular
			$\rho(J_q(a,0))^2=qa$. The value is nondecreasing in each of $a$ and $b$.
			\item[\rm(ii)] If $a\ge b\ge0$ and $2a+b\le q$, then
			$\rho(J_q(a,b))^2\le\tfrac12q(2a+b)-\tfrac{3bq^2}{16(q+b)}$.
			\item[\rm(iii)] For $1\le p\le N-1$, the graph $S_{N,p}$ has exactly one
			positive eigenvalue $\lambda_0$, which satisfies
			$
			\lambda_0^2-(p-1)\lambda_0-p(N-p)=0.
			$
			Every other eigenvalue is nonpositive.
			If $u$ is its positive unit Perron vector, then
			$
			\min_i u_i^2
			=
			\left(\frac{\lambda_0^2}{p}+N-p\right)^{-1}.
			$
		\end{enumerate}
	\end{lem}
	
	\begin{proof}
		For (i), vectors constant on each block reduce the biadjacency matrix of
		$J_q(a,b)$ to
		$\left(\begin{smallmatrix}\sqrt{ab}&\sqrt{a(q-b)}\\\sqrt{(q-a)b}&0\end{smallmatrix}\right)$,
		all other singular values being zero; the trace and determinant of the product of
		this matrix with its transpose are $T$ and $D$. Monotonicity holds because
		enlarging either block adds edges.
		
		For (ii), set $x=a/q$, $y=b/q$, $s=x+y-xy$ and $\Delta=s^2-4xy(1-x)(1-y)$, so that
		the normalized value is $\lambda=\tfrac12(s+\sqrt\Delta)$. Expanding,
		\begin{equation}\label{eq:blockid}
			(x+xy)^2-\Delta=y\bigl(2x-y+2xy(2x-1)\bigr),
		\end{equation}
		an identity that also appears in the proof of~\cite[Lemma~3.2]{LY}.
		If $x=0$ then $y=0$ and there is nothing to prove, so assume $x>0$; the hypotheses
		give $0\le y\le x\le\tfrac12$. For fixed $x$ the bracket in \eqref{eq:blockid} is
		decreasing in $y$, its coefficient of $y$ being $-1+2x(2x-1)<0$, so
		$2x-y+2xy(2x-1)\ge x(4x^2-2x+1)\ge3x/4$. In particular $\sqrt\Delta\le x+xy$, and
		rationalizing,
		\[
		\frac{2x+y}2-\lambda=\frac{x+xy-\sqrt\Delta}2
		=\frac{y\bigl(2x-y+2xy(2x-1)\bigr)}{2\bigl(x+xy+\sqrt\Delta\bigr)}
		\ \ge\ \frac{3xy/4}{4x(1+y)}=\frac{3y}{16(1+y)} .
		\]
		Multiplying by $q^2$ gives (ii).
		
		For (iii), the quotient matrix of $S_{N,p}$ with respect
		to the clique and the independent set is
		\(
		\begin{pmatrix}
			p-1&N-p\\
			p&0
		\end{pmatrix}.
		\)
		Its characteristic polynomial is the stated quadratic.
		Since $1\le p\le N-1$, its determinant is negative, so its
		two eigenvalues have opposite signs.
		Every remaining eigenvalue belongs to $\{-1,0\}$.
		
		The positive unit Perron vector is constant on each part.
		Write $\alpha$ for its value on the clique and $\beta$
		for its value on the independent set.
		The eigenvalue equation at an independent vertex gives
		$\lambda_0\beta=p\alpha$.
		Since $S_{N,p}$ contains $K_{p+1}$, spectral monotonicity
		gives $\lambda_0\ge p$, and hence
		$\alpha=(\lambda_0/p)\beta\ge\beta$.
		The normalization yields
		\[
		1=p\alpha^2+(N-p)\beta^2
		=\left(\frac{\lambda_0^2}{p}+N-p\right)\beta^2.
		\]
		Consequently,
		\[
		\min_i u_i^2=\beta^2
		=\left(\frac{\lambda_0^2}{p}+N-p\right)^{-1}.
		\]
	\end{proof}
	
	Finally, both extremal analyses end with the same combinatorial lemma.
	
	\begin{lem}\label{lem:box}
		Let $F$ be a $3$-partite $3$-graph with classes of the same size $t\ge16$ in which
		every vertex lies in at most $t^2/16$ missing transversal triples. Then $F$ has a
		perfect matching.
	\end{lem}
	
	\begin{proof}
		Let $M=\{\{x_i,y_i,z_i\}:i\in[m]\}$ be a maximum matching, with $x_i,y_i,z_i$ in
		the first, second and third class, and put $s=t-m$. If $s>0$, choose uncovered
		$x,y,z$, one in each class; no triple of uncovered vertices is an edge, so $x$ lies
		in at least $s^2$ missing triples and $s^2\le t^2/16$, giving $m\ge3t/4$. For
		distinct $i,j,k$ consider
		\[
		\{x,y_i,z_j\},\qquad \{x_i,y,z_k\},\qquad \{x_j,y_k,z\},\qquad \{x_k,y_j,z_i\},
		\]
		which are disjoint and cover the $i$th, $j$th and $k$th edges of $M$ together with
		$x,y,z$. Each of the four requirements fails for at most $t^3/16$ ordered triples
		$(i,j,k)$, whereas $m(m-1)(m-2)\ge12\cdot11\cdot10\,(t/16)^3>t^3/4$ for $t\ge16$.
		Replacing three edges of $M$ by four contradicts maximality.
	\end{proof}
	
	For a hypergraph $G$, write
	$\Delta_2(G)=\max_{u\ne v}|\{e\in E(G):\{u,v\}\subseteq e\}|$
	for its maximum codegree, where the maximum is taken over
	distinct vertices $u,v\in V(G)$.
	We also use the Chernoff bounds in their standard forms, and the nibble theorem of
	Frankl and R\"odl~\cite{FR} in the following form: for every integer $k\ge2$ and
	every $\zeta>0$ there are $\tau>0$ and $d_0$ such that, if $G$ is a $k$-uniform
	hypergraph on $N$ vertices and $N\ge D\ge d_0$ satisfies $(1-\tau)D<d_G(v)<(1+\tau)D$
	for every vertex and $\Delta_2(G)<\tau D$, then $G$ has a matching leaving at most
	$\zeta N$ vertices uncovered.
	
	\section{Fractional stability}\label{sec:fsnon}
	
	Both settings need the same passage: from the absence of a perfect fractional
	matching to a barrier. The partite case rests on a single spectral inequality for
	bipartite graphs carrying an asymmetric fractional cover, which we prove first.
	Lu and Yuan~\cite[Lemma~3.2]{LY} showed that these hypotheses, with strict
	inequality in the second condition, give $\rho(G)<(1/\sqrt2+\gamma)q$ for every
	$\gamma>0$ and every large even $q$; Lemma~\ref{lem:bipineq} replaces this by the exact
	bound, determines the equality cases, and adds a stability version.
	
	\begin{lem}\label{lem:bipineq}
		Let $G$ be a bipartite graph with classes $X$ and $Y$ of size $q$ admitting a
		fractional vertex cover $g$ with
		\begin{equation}\label{eq:asym}
			g(X)\ge g(Y),\qquad 2g(X)+g(Y)\le q .
		\end{equation}
		Then $\rho(G)\le q/\sqrt2$, with equality precisely when $q$ is even and, after
		possibly interchanging the classes, $E(G)=S\times Y$ for some $S\subseteq X$ of
		size $q/2$. Moreover, for every $\varepsilon>0$ there is $\delta>0$ such that
		\eqref{eq:asym} and $\rho(G)\ge(1/\sqrt2-\delta)q$ force, after possibly
		interchanging the classes, the existence of $S\subseteq X$ with
		\begin{equation}\label{eq:bipstab}
			\bigl||S|-q/2\bigr|\le\varepsilon q,\qquad
			\bigl|E(G)\,\triangle\,(S\times Y)\bigr|\le\varepsilon q^2 .
		\end{equation}
	\end{lem}
	
	\begin{proof}
		\emph{Reduction to few distinct weights.} The polytope
		$P=\{z\in[0,1]^{X\cup Y}: z(x)+z(y)\ge1\ (xy\in E(G)),\ z(X)\ge z(Y)\}$ is nonempty
		and compact; choose an extreme point $z$ of $P$ minimizing $2z(X)+z(Y)$, which
		exists because the minimizers form a nonempty face. Feasibility of $g$ gives
		\begin{equation}\label{eq:zsum}
			2z(X)+z(Y)\le q .
		\end{equation}
		Let $F=\{v:0<z(v)<1\}$ and retain on $F$ the edges $xy$ with $z(x)+z(y)=1$. On each
		component of the resulting graph meeting both classes, following a path of tight
		edges shows that all weights in $X$ equal some $t\in(0,1)$ and all weights in $Y$
		equal $1-t$. Adding $h$ to the $X$-weights of one component and subtracting $h$
		from its $Y$-weights preserves every tight equation, and an isolated fractional
		vertex may be perturbed by itself; for $h$ small of either sign the box constraints
		and the slack inequalities persist, and no tight edge joins a fractional vertex to
		an integral one. If $z(X)>z(Y)$ then any such perturbation is feasible in both
		directions, contradicting extremality; and if $z(X)=z(Y)$ and there were two
		components, a nonzero combination of their perturbations preserving
		$z(X)-z(Y)=0$ would do the same. A single isolated fractional vertex is also
		excluded, since $z(X)=z(Y)$ would then equate a fractional weight with an integer.
		So either $z$ is integral, or $z(X)=z(Y)$ and the weights lie in $\{0,t,1\}$ on $X$
		and in $\{0,1-t,1\}$ on $Y$ for one $t\in(0,1)$.
		
		\emph{The integral case.} Put $a=z(X)/q$ and $b=z(Y)/q$. Every edge meets a vertex
		of weight one, so $G\subseteq J_q(aq,bq)$, and \eqref{eq:zsum} gives $a\ge b$ and
		$2a+b\le1$. By Lemma~\ref{lem:blocks}(ii),
		\begin{equation}\label{eq:intcase}
			\frac{\rho(G)^2}{q^2}\ \le\ \frac{2a+b}2-\frac{3b}{16(1+b)}\ \le\ \frac12 ,
		\end{equation}
		with equality throughout only for $a=1/2$ and $b=0$.
		
		\emph{The fractional case.} Let the normalized sizes of the weight groups $1,t,0$
		in $X$ be $a,p,1-a-p$ and those of $1,1-t,0$ in $Y$ be $b,r,1-b-r$, so that
		\begin{equation}\label{eq:means}
			a+pt=b+r(1-t)=s\le\tfrac13 .
		\end{equation}
		The cover inequalities place $G$ inside the template whose normalized matrix is
		\begin{equation}\label{eq:Cmat}
			C(a,p,b,r)=
			\begin{pmatrix}
				\sqrt{ab} & \sqrt{ar} & \sqrt{a(1-b-r)}\\
				\sqrt{pb} & \sqrt{pr} & 0\\
				\sqrt{(1-a-p)b} & 0 & 0
			\end{pmatrix},
		\end{equation}
		a block of ones between groups of normalized masses $u$ and $v$ contributing the
		entry $\sqrt{uv}$; vectors constant on groups account for all nonzero singular
		values, so the template has spectral radius $q\|C\|$.
		
		We may raise the common mean in \eqref{eq:means} to $1/3$. Put
		$\theta=(1/3-s)/(1-s)$ and split off the proportion $\theta$ of every weight group
		on both sides, raising its weight to one. Splitting a group into two copies leaves
		the norm unchanged, since the original coordinate embeds isometrically with
		coefficients the square roots of the relative masses and the complement is
		annihilated; raising weights only adds allowed pairs, and increasing a nonnegative
		matrix entrywise cannot decrease its operator norm. The new parameters are
		$a'=\theta+(1-\theta)a$, $p'=(1-\theta)p$, $b'=\theta+(1-\theta)b$,
		$r'=(1-\theta)r$, with both means equal to $1/3$. It therefore suffices to treat
		$a+pt=b+r(1-t)=1/3$. Transposing exchanges $t$ and $1-t$, so assume
		$0<t\le\tfrac12$. Then $p=(1/3-a)/t$ and $r=(1/3-b)/(1-t)$, and nonnegativity of
		the six masses amounts to
		\begin{equation}\label{eq:rect}
			L_t\le a\le\tfrac13,\qquad 0\le b\le\tfrac13,\qquad
			L_t=\max\Bigl\{0,\frac{1/3-t}{1-t}\Bigr\},
		\end{equation}
		the condition $1-b-r\ge0$ reducing to $2/3-t+tb\ge0$, which holds for $t\le1/2$.
		
		For fixed $t$ the function $\|C\|^2$ is separately convex in $a$ and in $b$ on the
		rectangle \eqref{eq:rect}. Indeed, writing
		$C=\mathrm{diag}(\sqrt a,\sqrt p,\sqrt{1-a-p})\,K\,
		\mathrm{diag}(\sqrt b,\sqrt r,\sqrt{1-b-r})$ with $K$ the matrix
		whose entries are $1$ on and above the antidiagonal and $0$ below it, the matrix $C^{\top}C$ is affine in $a$
		for fixed $b$, because the row masses are; and the largest eigenvalue of a
		symmetric matrix is convex in an affine parameter. The same argument applied to
		$CC^{\top}$ gives convexity in $b$. Applying the two inequalities in turn bounds
		$\|C\|^2$ by the bilinear interpolation of its four corner values, so the maximum
		is at a corner.
		
		If $\tfrac13\le t\le\tfrac12$, then $L_t=0$ and the corner values are
		\begin{equation}\label{eq:corners}
			\|C\|^2=\frac1{9t(1-t)},\ \frac13,\ \frac13,\ \frac49
			\qquad\text{at}\quad (a,b)=(0,0),\ (\tfrac13,0),\ (0,\tfrac13),\ (\tfrac13,\tfrac13).
		\end{equation}
		At $(0,0)$ the only nonzero block joins the two fractional groups, of masses
		$1/(3t)$ and $1/(3(1-t))$; at $(\tfrac13,0)$ there is only a full row block of mass
		$1/3$, and the transpose at $(0,\tfrac13)$; at $(\tfrac13,\tfrac13)$ the template is
		$J_q(q/3,q/3)$, of squared norm $4/9$ by Lemma~\ref{lem:blocks}(i). Since
		$t(1-t)\ge2/9$ on this interval, every corner value is at most $1/2$, attained only
		at $(0,0)$ with $t=1/3$; and the interpolation puts positive weight on another
		corner as soon as $a>0$ or $b>0$. So equality forces
		\begin{equation}\label{eq:eqcase}
			t=\tfrac13,\qquad a=b=0,\qquad p=1,\qquad r=\tfrac12 .
		\end{equation}
		If $0<t<\tfrac13$, put $r_0=1/(3(1-t))$, so that $\tfrac13<r_0<\tfrac12$ and
		$L_t=1-2r_0$. At $(L_t,0)$ the $X$ side has no zero-weight group, so the template
		is $J_q$ with active proportions $L_t$ and $r_0$, which after transposing satisfy
		$r_0\ge L_t>0$ and $2r_0+L_t=1$; Lemma~\ref{lem:blocks}(ii), with its strictly
		positive deficit, gives squared norm strictly below $1/2$. At $(L_t,\tfrac13)$ the
		template is contained in $J_q(q/3,q/3)$ and has squared norm at most $4/9$, and the
		corners $(\tfrac13,0)$ and $(\tfrac13,\tfrac13)$ give $1/3$ and $4/9$. Transposition
		covers $\tfrac12<t<1$, with the equality case $t=2/3$, $a=b=0$, $p=1/2$, $r=1$.
		
		Returning to a mean $s\le1/3$: if the original template has norm $1/\sqrt2$, so does
		the enlarged one, whose equality cases have $a'=b'=0$; since $a',b'\ge\theta$ this
		forces $\theta=0$ and $s=1/3$. Every equality template is thus a complete rectangle
		with one side of normalized size $1/2$ and the other of size $1$.
		
		\emph{Equality.} If $\rho(G)=q/\sqrt2$, its containing template is such a rectangle
		and has exactly $q^2/2$ edges, so $q^2/2=\rho(G)^2\le e(G)\le q^2/2$ by
		Lemma~\ref{lem:basic}(ii); no template edge is missing and the half-sized group is
		an actual vertex set, so $q$ is even. Conversely either orientation admits a cover
		satisfying \eqref{eq:asym}: weight $1$ on $S$ and $0$ elsewhere for $S\times Y$, and
		weight $1/3$ on $X$, $2/3$ on $S$ and $0$ on $Y\setminus S$ for $X\times S$.
		
		\emph{Stability.} Suppose \eqref{eq:bipstab} fails for some $\varepsilon>0$. Then
		there are $G_j$ with class size $q_j$ satisfying \eqref{eq:asym} and
		$\rho(G_j)/q_j\to1/\sqrt2$, for which no rectangle works. Take the extreme cover
		above and its template $T_j$, and pass to a subsequence in which all covers fall in
		the same alternative. In the integral case the parameters lie in the compact set
		$\{(a,b):a\ge b\ge0,\ 2a+b\le1\}$; along a convergent subsequence
		\eqref{eq:intcase} forces the limit $(1/2,0)$, and with $S_j$ the weight-one subset
		of the first class, $|S_j|/q_j\to1/2$ and
		$|E(T_j)\triangle(S_j\times Y_j)|=o(q_j^2)$, the only extra template edges meeting
		the weight-one subset of the second class, of size $o(q_j)$. In the fractional case
		the vector $(t,a,p,b,r)$ lies in the compact set cut out by $0\le t\le1$,
		$a,p,b,r\ge0$, $a+p\le1$, $b+r\le1$ and $a+pt=b+r(1-t)\le1/3$, and
		\eqref{eq:Cmat} is continuous there; a convergent subsequence has limiting norm
		$1/\sqrt2$. The limit cannot have $t=0$, since then $a=b+r=s\le1/3$ and the
		template lies in $J_q$ with active proportions $s,s$, of norm at most $2/3$; nor
		$t=1$, by the same bound. The limit therefore has $0<t<1$.
		By the equality characterization above, after interchanging
		the two classes throughout the subsequence if necessary,
		we have $t=1/3$, $a=b=0$, $p=1$ and $r=1/2$.
		This interchange preserves the relevant constraints because
		the two class sums of the chosen extreme cover are equal. With $S_j$ the
		fractional group in the second class, of normalized size $r_j\to1/2$, the blocks
		meeting weight-one groups contribute at most $(a_j+b_j)q_j^2=o(q_j^2)$ extra edges
		and the missing part of $X_j\times S_j$ has at most $(1-p_j)r_jq_j^2=o(q_j^2)$
		edges. In either alternative there is a rectangle $R_j$ with
		$|E(T_j)\triangle R_j|=o(q_j^2)$ and $e(T_j)=q_j^2/2+o(q_j^2)$, and
		$0\le e(T_j)-e(G_j)\le e(T_j)-\rho(G_j)^2=o(q_j^2)$ by containment and
		Lemma~\ref{lem:basic}(ii). Hence $|E(G_j)\triangle R_j|=o(q_j^2)$, a contradiction.
	\end{proof}
	
	The partite stability theorem follows by applying Lemma~\ref{lem:bipineq} to the
	link of a vertex of cover weight zero in the heaviest class and then propagating the
	conclusion from that link to the whole hypergraph.
	
	\begin{thm}\label{thm:target}
		For every $\varepsilon>0$ there are $\delta>0$ and $q_0$ such that the following
		holds for $q\ge q_0$. Let $H$ be a $q$-balanced $3$-partite $3$-graph with
		$\sigma(H)\ge(\tfrac1{\sqrt2}-\delta)q$ and with no perfect fractional matching,
		and let $g$ be any minimum fractional vertex cover. Then the classes can be
		labelled so that there are $A_1\subseteq V_1$, $A_2\subseteq V_2$ and
		$D\subseteq V_3$ with
		\[
		\bigl||A_i|-q/2\bigr|\le\varepsilon q\ \ (i=1,2),\qquad |D|\le\varepsilon q,
		\]
		\[
		\sum_{v\in V_1}\bigl|g(v)-\mathbf1_{A_1}(v)\bigr|
		+\sum_{v\in V_2}\bigl|g(v)-\mathbf1_{A_2}(v)\bigr|
		+\sum_{v\in V_3}g(v)\ \le\ \varepsilon q ,
		\]
		and every edge of $H$ meets $A_1\cup A_2\cup D$. In particular every edge of $H-D$
		meets $A_1\cup A_2$.
	\end{thm}
	
	The proof needs one compactness statement, which recovers individual cover weights
	from their means.
	
	\begin{lem}\label{lem:recover}
		Let $q_j\to\infty$, let $G_j$ be bipartite with classes $U_j,W_j$ of size $q_j$, and
		let $f_j:U_j\to[0,1]$ and $h_j:W_j\to[0,1]$ satisfy $f_j(u)+h_j(w)\ge1$ for every
		$uw\in E(G_j)$. Fix $a>0$. If $\rho(G_j)\ge(\tfrac1{\sqrt2}-o(1))q_j$,
		$\sum_{u}f_j(u)=(\tfrac a2+o(1))q_j$ and
		$\sum_{w}|h_j(w)-(1-a)|=o(q_j)$, then there are $R_j\subseteq U_j$ with
		$|R_j|=(\tfrac12+o(1))q_j$ and $\sum_u|f_j(u)-a\mathbf1_{R_j}(u)|=o(q_j)$.
	\end{lem}
	
	\begin{proof}
		We suppress $j$ and write $q=q_j$, all asymptotic statements referring to the
		sequence. Fix $0<\varepsilon\le a/2$ and put
		\[
		R_\varepsilon=\{u\in U:f(u)\ge a-\varepsilon\},\qquad
		W_\varepsilon=\{w\in W:h(w)>1-a+\varepsilon\} .
		\]
		Since $\sum_w|h(w)-(1-a)|=o(q)$, we have $|W_\varepsilon|\le o(q)/\varepsilon=o(q)$
		for fixed $\varepsilon$. If $u\notin R_\varepsilon$ and $uw\in E(G)$, then
		$h(w)\ge1-f(u)>1-a+\varepsilon$ and $w\in W_\varepsilon$; hence every vertex outside
		$R_\varepsilon$ has degree at most $|W_\varepsilon|$, and
		$e(G)\le|R_\varepsilon|q+q|W_\varepsilon|=|R_\varepsilon|q+o(q^2)$. As
		$e(G)\ge\rho(G)^2\ge(\tfrac12-o(1))q^2$ by Lemma~\ref{lem:basic}(ii), this gives
		$|R_\varepsilon|\ge(\tfrac12-o(1))q$. In the other direction
		$(\tfrac a2+o(1))q=\sum_uf(u)\ge(a-\varepsilon)|R_\varepsilon|$, and
		$\tfrac a{2(a-\varepsilon)}\le\tfrac12+\tfrac\varepsilon a$, so
		\begin{equation}\label{eq:Rsize}
			\bigl||R_\varepsilon|-q/2\bigr|\ \le\ \Bigl(\frac\varepsilon a+o(1)\Bigr)q .
		\end{equation}
		Using $|R_\varepsilon|\ge(\tfrac12-o(1))q$ in the first estimate and
		\eqref{eq:Rsize} in the second,
		\begin{align*}
		\sum_{u\notin R_\varepsilon}f(u)&\le\sum_uf(u)-(a-\varepsilon)|R_\varepsilon|
		\le\Bigl(\frac\varepsilon2+o(1)\Bigr)q,\\
		\sum_{u\in R_\varepsilon}\bigl(f(u)-a\bigr)&\le\sum_uf(u)-a|R_\varepsilon|
		\le(\varepsilon+o(1))q ,
		\end{align*}
		while $\sum_{u\in R_\varepsilon}(a-f(u))_+\le\varepsilon|R_\varepsilon|\le\varepsilon q$.
		Since $|x|=x+2(-x)_+$,
		\begin{equation}\label{eq:Rell1}
			\sum_u\bigl|f(u)-a\mathbf1_{R_\varepsilon}(u)\bigr|
			=\sum_{u\in R_\varepsilon}|f(u)-a|+\sum_{u\notin R_\varepsilon}f(u)
			\ \le\ \Bigl(\frac{7\varepsilon}2+o(1)\Bigr)q .
		\end{equation}
		Finally let $\varepsilon$ tend to zero along the sequence: for each integer
		$i\ge2/a$ choose $J_i$ so that the error terms in \eqref{eq:Rsize} and
		\eqref{eq:Rell1} with $\varepsilon=1/i$ are at most $q_j/i$ for $j\ge J_i$, with
		$J_2<J_3<\cdots$, and put $R_j=R_{1/i}$ for $J_i\le j<J_{i+1}$. Then
		$|R_j|=(\tfrac12+o(1))q_j$ and $\sum_u|f_j(u)-a\mathbf1_{R_j}(u)|=o(q_j)$.
	\end{proof}
	
	\begin{proof}[Proof of Theorem~\ref{thm:target}]
		Suppose the assertion fails for some $\varepsilon>0$. For each $j$ there is a
		counterexample $H$ with class size $q\ge j$ and
		$\sigma(H)\ge(\tfrac1{\sqrt2}-\tfrac1j)q$, together with a minimum cover $g$
		witnessing the failure; we suppress $j$, and every asymptotic statement refers to
		this sequence, passing to subsequences freely.
		
		Let $f$ be a maximum fractional matching, and put $d_f(v)=\sum_{e\ni v}f(e)$ for each vertex $v$. Writing
		\[
		0=\sum_vg(v)-\sum_ef(e)
		=\sum_vg(v)\bigl(1-d_f(v)\bigr)+\sum_ef(e)\Bigl(\sum_{v\in e}g(v)-1\Bigr)
		\]
		exhibits zero as a sum of nonnegative terms, and in each class
		$\sum_{v\in V_i}(1-d_f(v))=q-\nu^*(H)>0$, so some $u_i\in V_i$ has $d_f(u_i)<1$ and
		hence $g(u_i)=0$. Thus the restriction of $g$ to the other two classes covers
		$L_H(u_i)$, for each $i$; we keep this cover throughout. Relabel so that
		$m_i=q^{-1}g(V_i)$ satisfies $m_1\ge m_2\ge m_3$, whence $2m_2+m_3<1$.
		Lemma~\ref{lem:bipineq} applied to $L_H(u_1)$ shows that this link differs in
		$o(q^2)$ edges from a complete bipartite graph with one class of size
		$(\tfrac12+o(1))q$ and the other equal to $V_2$ or to $V_3$; pass to a subsequence
		in which the orientation is fixed.
		
		\emph{The half-sized class lies in $V_2$.} Let $S\subseteq V_2$ with
		$|S|=(\tfrac12+o(1))q$ be such that $L_H(u_1)$ differs from $K_{S,V_3}$ in $o(q^2)$
		edges, and put
		$p=|S|/q$ and $A=q^{-1}g(S)$. Summing the cover inequality over the edges of the
		link inside $S\times V_3$ and using $0\le g\le1$ gives $A+pm_3\ge p-o(1)$, so
		\begin{equation}\label{eq:squeeze}
			1-o(1)\ \le\ 2A+m_3\ \le\ 2m_2+m_3\ \le\ m_1+m_2+m_3\ <\ 1 .
		\end{equation}
		Consecutive differences therefore vanish, so $m_2-A=o(1)$, $m_1-m_2=o(1)$ and
		$m_1+m_2+m_3=1-o(1)$; after a further subsequence there is $a\in[\tfrac23,1]$ with
		$m_1,m_2\to a/2$ and $m_3\to1-a$, the bounds on $a$ coming from $m_2\ge m_3\ge0$.
		
		To pass from means to individual weights, choose $y\in S$ and $z\in V_3$
		independently and uniformly and put $F(y,z)=g(y)+g(z)-1$. Its negative part has
		expectation $o(1)$, since $F\ge0$ on link edges and the proportion of missing pairs
		is $o(1)$, while $\mathbb EF=A/p+m_3-1=o(1)$; so $\mathbb E|F|=o(1)$ by
		$|F|=F+2\max\{-F,0\}$. Averaging over $z$ first gives
		$|g(y)+m_3-1|\le\mathbb E_z|F(y,z)|$ for each $y\in S$, and then averaging over $y$,
		together with $m_3\to1-a$, gives $\sum_{y\in S}|g(y)-a|=o(q)$; reversing the roles
		gives $\sum_{z\in V_3}|g(z)-(1-a)|=o(q)$. Since $m_2-A=o(1)$, the weight on
		$V_2\setminus S$ is $o(q)$, so
		\begin{equation}\label{eq:V2V3}
			\sum_{y\in V_2}\bigl|g(y)-a\mathbf1_S(y)\bigr|=o(q),\qquad
			\sum_{z\in V_3}\bigl|g(z)-(1-a)\bigr|=o(q).
		\end{equation}
		Applying Lemma~\ref{lem:recover} to $L_H(u_2)$ with $U=V_1$ and $W=V_3$ gives
		$R\subseteq V_1$ with $|R|=(\tfrac12+o(1))q$ and
		$\sum_{x\in V_1}|g(x)-a\mathbf1_R(x)|=o(q)$.
		
		We claim $a=1$. If not, put $b=(1-a)/4>0$; all but $o(q)$ vertices of $V_1\cup V_2$
		are within $b$ of their target weights, and for a pair $(x,y)\notin R\times S$ with
		both ends unexceptional one target is zero and the other at most $a$, so
		$g(x)+g(y)\le a+2b=(1+a)/2<1$ and the pair is not an edge of $L_H(u_3)$. Hence
		$e(L_H(u_3))\le|R||S|+o(q^2)=(\tfrac14+o(1))q^2$, contradicting
		$e(L_H(u_3))\ge\rho(L_H(u_3))^2\ge(\tfrac12-o(1))q^2$. So $a=1$ and
		\begin{equation}\label{eq:ell1}
			\sum_{x\in V_1}\bigl|g(x)-\mathbf1_R(x)\bigr|
			+\sum_{y\in V_2}\bigl|g(y)-\mathbf1_S(y)\bigr|+\sum_{z\in V_3}g(z)=o(q).
		\end{equation}
		
		\emph{The half-sized class lies in $V_3$.} Now $S\subseteq V_3$ has size
		$(\tfrac12+o(1))q$ and $L_H(u_1)$ differs from $K_{V_2,S}$ in $o(q^2)$ edges. With
		$p=|S|/q$ and $C=q^{-1}g(S)$, summing the cover inequalities over the rectangle
		gives $pm_2+C\ge p-o(1)$, so
		$1-o(1)\le m_2+2C\le m_2+2m_3\le2m_2+m_3\le m_1+m_2+m_3<1$; hence
		$m_1,m_2,m_3\to\tfrac13$ and $m_3-C=o(1)$. The averaging argument on $V_2\times S$
		gives $\sum_{y\in V_2}|g(y)-\tfrac13|=o(q)$ and
		$\sum_{z\in V_3}|g(z)-\tfrac23\mathbf1_S(z)|=o(q)$, and Lemma~\ref{lem:recover}
		applied to $L_H(u_3)$ with $U=V_1$, $W=V_2$ and $a=\tfrac23$ gives $R\subseteq V_1$
		of size $(\tfrac12+o(1))q$ with $\sum_{x}|g(x)-\tfrac23\mathbf1_R(x)|=o(q)$. In
		$L_H(u_2)$ a pair outside $R\times S$ whose weights deviate from these targets by at
		most $\tfrac1{12}$ has weight sum at most $\tfrac23+\tfrac16<1$ and is not an edge,
		and there are only $o(q)$ exceptional vertices, so
		$e(L_H(u_2))\le(\tfrac14+o(1))q^2$, again contradicting the spectral lower bound.
		This orientation is impossible.
		
		Finally set $A_i=\{v\in V_i:g(v)\ge\tfrac13\}$ for $i=1,2$ and
		$D=\{v\in V_3:g(v)\ge\tfrac13\}$. For $T=R$ in $V_1$ or $T=S$ in $V_2$, a vertex of
		$A_i\setminus T$ contributes at least $\tfrac13$ to
		$\sum_{v\in V_i}|g(v)-\mathbf1_T(v)|$ and a vertex of $T\setminus A_i$ more than
		$\tfrac23$, so $|A_1\triangle R|=o(q)$ and $|A_2\triangle S|=o(q)$; and
		$|D|\le3g(V_3)=o(q)$. With \eqref{eq:ell1} this gives the two displayed estimates
		for large members of the sequence. An edge avoiding $A_1\cup A_2\cup D$ would have
		all three weights below $\tfrac13$, contradicting the cover inequality. This
		contradicts the choice of counterexamples.
	\end{proof}

	\medskip\noindent\emph{The non-partite case.}
	For a ground set $V$ and $A\subseteq V$ put
	$T_i(A)=\{e\in\binom V3:|e\cap A|\ge i\}$ for $i\in\{1,2\}$. If $|V|=3m$ and
	$|A|=im-1$, then no matching in $T_i(A)$ has $m$ edges; these are the two space
	barriers.
	
	\begin{thm}\label{thm:fsnon}
		For every $\beta>0$ there are $\delta\in(0,10^{-4}]$ and an integer $N_0$ such that
		the following holds. If $F$ is a $3$-graph on $N\ge N_0$ vertices with
		$\sigma(F)\ge(\tfrac23-\delta)N$ and $F$ has no perfect fractional matching, then
		there are $i\in\{1,2\}$ and $A,D\subseteq V(F)$ with
		\[
		\Bigl||A|-\tfrac{iN}3\Bigr|\le\beta N,\qquad|D|\le\beta N,\qquad
		E(F-D)\subseteq T_i(A).
		\]
	\end{thm}
	
	The proof rests on two estimates. The first is the instance of
	Lemma~\ref{lem:missing} for the split graphs that arise from rounded fractional
	covers.
	
	\begin{lem}\label{lem:split}
		Suppose that $n\ge12$ and that $a$ is an integer with $n/4\le a\le n/3$. Let
		$T=K_a\vee I_{n-a}$, let $K$ be a spanning subgraph of $T$, and put
		$D=\rho(T)-\rho(K)$. Then $|E(T)\setminus E(K)|\le416Dn$.
	\end{lem}
	
	\begin{proof}
		Write $\lambda=\rho(T)$ and let $\mathbf u$ be its positive unit Perron vector. By
		symmetry $\mathbf u$ has a constant value $\alpha$ on the clique and $\beta$ on the
		independent set, and the eigenvalue equation at a vertex of the independent set
		gives $\lambda\beta=a\alpha$. Since $T$ contains $K_{a,n-a}$ and $n-a\ge2a$,
		\[
		\lambda\ \ge\ \sqrt{a(n-a)}\ >\ a,\qquad
		\lambda\ \ge\ \frac{\sqrt3}4\,n\ >\ \frac n3 ,
		\]
		the second bound because $a(n-a)\ge\tfrac n4\cdot\tfrac{3n}4=\tfrac3{16}n^2$. Thus
		$\alpha\ge\beta$, and $1=a\alpha^2+(n-a)\beta^2\le n\alpha^2$. As $\lambda\le n$ and
		$a\ge n/4$,
		\[
		\alpha\ge n^{-1/2},\qquad \beta=\frac a\lambda\alpha\ge\frac\alpha4,\qquad
		\min_i u_i^2\ \ge\ \frac1{16n} .
		\]
		The quotient matrix of $T$ with respect to the clique and the independent set is
		$\bigl(\begin{smallmatrix}a-1&n-a\\a&0\end{smallmatrix}\bigr)$, whose determinant is
		negative, so its two eigenvalues have opposite signs; the remaining eigenvalues of
		$T$ are $-1$ and $0$. Every eigenvalue other than $\lambda$ is therefore
		nonpositive, and Lemma~\ref{lem:missing} with $N=n$ and the bound just obtained on $\min_iu_i^2$ gives
		$|E(T)\setminus E(K)|\le16nD(8n/\lambda+2)\le16nD(24+2)=416Dn$.
	\end{proof}
	
	On an $n$-vertex $3$-graph the maximum size of a fractional matching is at most
	$n/3$, since summing the vertex constraints gives $3\sum_ef(e)\le n$, with equality
	exactly when every vertex constraint is tight. By Lemma~\ref{lem:basic}(i) the absence
	of a perfect fractional matching is therefore equivalent to the existence of a
	fractional vertex cover of total weight less than $n/3$; no divisibility assumption
	on $n$ is needed.
	
	\begin{lem}\label{lem:dichot}
		Let $H$ be a $3$-graph on $n$ vertices, where
		\[
		0<\delta\le10^{-4},\qquad n\ge\max\{12,\delta^{-1/2}\},\qquad
		\sigma(H)\ge\Bigl(\tfrac23-\delta\Bigr)n .
		\]
		If $H$ has no perfect fractional matching, then at least one of the following
		holds.
		\begin{enumerate}
			\item[(i)] $H$ has a vertex cover $A'$ with
			$\bigl(\tfrac13-27\delta\bigr)n\le|A'|\le\bigl(\tfrac13+210000\delta\bigr)n$.
			\item[(ii)] There are $A,D\subseteq V(H)$ with
			$\bigl(\tfrac23-109\delta\bigr)n\le|A|\le\tfrac{2n}3$ and
			$|D|\le64\sqrt\delta\,n$ such that every edge of $H-D$ contains at least two
			vertices of $A$.
		\end{enumerate}
	\end{lem}
	
	\begin{proof}
		Choose a fractional vertex cover $w$ with $\sum_vw_v<n/3$ and put
		$a_0=\min_vw_v<\tfrac13$. The coordinates $\widetilde z_v=(w_v-a_0)/(1-3a_0)$ are
		nonnegative with minimum zero, satisfy $\sum_v\widetilde z_v<n/3$, and satisfy
		$\sum_{v\in e}\widetilde z_v\ge1$ for every edge. Setting
		$z_v=\min\{\widetilde z_v,1\}$ preserves the cover condition, since an edge
		containing a truncated coordinate already receives weight $1$ from it. Hence
		\begin{equation}\label{eq:cover}
			0\le z_v\le1,\qquad\min_vz_v=0,\qquad\sum_vz_v<\frac n3,\qquad
			\sum_{v\in e}z_v\ge1\quad(e\in E(H)).
		\end{equation}
		
		Choose $u$ with $z_u=0$ and let $G$ be $L_H(u)$ together with $u$ as an isolated
		vertex, a graph on $n$ vertices; every edge $xy$ of $G$ satisfies $z_x+z_y\ge1$.
		For $0<t<\tfrac12$ define $f_t(z)=0$ if $z<t$, $f_t(z)=\tfrac12$ if
		$t\le z\le1-t$, and $f_t(z)=1$ if $z>1-t$. If $z_x+z_y\ge1$ then
		$f_t(z_x)+f_t(z_y)\ge1$: if one rounded value is $0$, say $z_x<t$, then
		$z_y\ge1-z_x>1-t$ and the other is $1$; otherwise both are at least $\tfrac12$.
		Taking $t$ uniformly from $(0,\tfrac12)$ gives $\mathbb E_tf_t(z)=z$ for every
		$z\in[0,1]$, so by \eqref{eq:cover} some $t$ has $\sum_vf_t(z_v)<n/3$. Let $X,Y,Z$
		be the sets on which the rounded value is $1,\tfrac12,0$ and put $a=|X|$, $b=|Y|$,
		$c=|Z|$. The rounded cover condition excludes edges between $Y$ and $Z$ and edges
		inside $Z$, so
		\begin{equation}\label{eq:J}
			G\ \subseteq\ J:=K_a\vee(K_b\cup I_c),\qquad a+b+c=n,\qquad 2a+b<\frac{2n}3 .
		\end{equation}
		
		Put $x=a/n$ and $y=b/n$, and add a loop at each vertex of $X\cup Y$ in $J$, which
		only increases the largest eigenvalue. On the subspace of vectors constant on each
		part the resulting matrix has symmetric quotient $nM(x,y)$, where
		\begin{equation}\label{eq:Mxy}
			M(x,y)=\begin{pmatrix}
				x&\sqrt{xy}&\sqrt{x(1-x-y)}\\
				\sqrt{xy}&y&0\\
				\sqrt{x(1-x-y)}&0&0\end{pmatrix},
		\end{equation}
		and the orthogonal complement lies in the kernel. Writing $r(x,y)$ for the largest
		eigenvalue of $M(x,y)$ we therefore have $\rho(G)/n\le r(x,y)$.
		
		For fixed $x\in[0,\tfrac13]$ the quantity $r(x,y)$ is nondecreasing on
		$0\le y\le\tfrac23-2x$. Indeed $r(0,y)=y$; and if $x>0$ and $y>0$ the matrix is
		irreducible, its largest eigenvalue $s$ exceeds $y$, and eliminating the last two
		coordinates from the eigenvalue equations gives
		\[
		\Phi(s,y):=s-x-\frac{xy}{s-y}-\frac{x(1-x-y)}{s}=0 .
		\]
		For $s>y$,
		\[
		\frac{\partial\Phi}{\partial s}=1+\frac{xy}{(s-y)^2}+\frac{x(1-x-y)}{s^2}>0,
		\qquad
		\frac{\partial\Phi}{\partial y}=-\frac{xs}{(s-y)^2}+\frac xs\le0 ,
		\]
		the second because $0<s-y\le s$; moreover $\Phi\to-\infty$ as $s\downarrow y$ and
		$\Phi\to\infty$ as $s\to\infty$. So the root is unique and nondecreasing in $y$,
		and continuity extends this to the endpoints. With \eqref{eq:J} this gives
		$r(x,y)\le r(x,\tfrac23-2x)=:r_x$.
		
		Put $y_x=\tfrac23-2x$ and $c_x=\tfrac13+x$. Expanding $\det(sI-M(x,y_x))$ gives
		$F_x(s)=s^3-(x+y_x)s^2-xc_xs+xy_xc_x$, and
		\begin{equation}\label{eq:Fx}
			F_x\Bigl(\frac23\Bigr)=\frac{2x}9(1-3x)(2+3x)=\frac23\,x\Bigl(\frac13-x\Bigr)(2+3x).
		\end{equation}
		For $0<x<\tfrac13$ this is positive. By interlacing with the principal submatrix
		$\mathrm{diag}(y_x,0)$, the second largest eigenvalue of $M(x,y_x)$ is at most
		$y_x<\tfrac23$, and the third is no larger; since the characteristic polynomial is
		the product of $s$ minus the three eigenvalues, positivity in \eqref{eq:Fx} forces
		$r_x<\tfrac23$. At $x=0$ the matrix is diagonal with largest eigenvalue $\tfrac23$,
		and at $x=\tfrac13$ its nonzero block
		$\bigl(\begin{smallmatrix}1/3&\sqrt2/3\\\sqrt2/3&0\end{smallmatrix}\bigr)$ has
		eigenvalues $\tfrac23$ and $-\tfrac13$. So $r_x\le\tfrac23$ throughout, and the
		spectral hypothesis gives
		\begin{equation}\label{eq:rx}
			\frac23-\delta\ \le\ r_x\ \le\ \frac23 .
		\end{equation}
		For $0\le s\le\tfrac23$, the bounds $x+y_x\le\tfrac23$ and $xc_x\le\tfrac29$ give
		$|F_x'(s)|\le3s^2+2(x+y_x)s+xc_x\le\tfrac43+\tfrac89+\tfrac29<6$. Since
		$F_x(r_x)=0$, the mean value theorem and \eqref{eq:rx} yield
		$\tfrac43x(\tfrac13-x)\le F_x(\tfrac23)\le6\delta$, where the first inequality uses
		\eqref{eq:Fx} and $2+3x\ge2$. Putting $d=\min\{x,\tfrac13-x\}\in[0,\tfrac16]$ and
		using $x(\tfrac13-x)=d(\tfrac13-d)\ge d/6$ gives $d\le27\delta$, that is
		\begin{equation}\label{eq:dichot}
			a\le27\delta n\qquad\text{or}\qquad a\ge\Bigl(\frac13-27\delta\Bigr)n .
		\end{equation}
		
		Suppose first that $a\le27\delta n$. Then $\rho(G)\ge n/2$ and $a\le n/8$ because
		$\delta\le10^{-4}$, so Lemma~\ref{lem:basic}(iv) gives
		$\rho(G-X)\ge\rho(G)-4a\ge(\tfrac23-109\delta)n$. By \eqref{eq:J} the vertices of
		$Z$ are isolated in $G-X$ and the remaining edges lie inside $Y$, so
		$\rho(G[Y])=\rho(G-X)$, and with $\rho(G[Y])\le b$ and \eqref{eq:J},
		\begin{equation}\label{eq:bbound}
			\Bigl(\frac23-109\delta\Bigr)n\ \le\ b\ \le\ \frac{2n}3 .
		\end{equation}
		Let $M$ be the number of edges of $K_Y$ missing from $G[Y]$. By
		Lemma~\ref{lem:basic}(ii), $2e(G[Y])\ge\rho(G[Y])^2$, so
		\begin{equation}\label{eq:M1}
			M=\binom b2-e(G[Y])\le\frac12\Bigl[\Bigl(\frac{2n}3\Bigr)^2
			-\Bigl(\Bigl(\frac23-109\delta\Bigr)n\Bigr)^2\Bigr]\le73\delta n^2 .
		\end{equation}
		Return now to the unrounded cover $z$. The set $L=\{v\in Y:z_v<\tfrac12\}$ is
		independent in $G$, since two of its vertices have total weight less than one, so
		$\binom{|L|}2\le M$ and \eqref{eq:M1} gives
		$|L|\le\sqrt{2M}+1\le\sqrt{146\delta}\,n+1\le14\sqrt\delta\,n$, using
		$\sqrt\delta\,n\ge1$. Writing $s_+=\max\{s,0\}$ and
		$d_-=\sum_{v\in Y}(\tfrac12-z_v)_+$, $d_+=\sum_{v\in Y}(z_v-\tfrac12)_+$, only
		vertices of $L$ contribute to $d_-$ and each contributes at most $\tfrac12$, so
		$d_-\le7\sqrt\delta\,n$. Since $\sum_{v\in Y}z_v=\tfrac b2+d_+-d_-$,
		\eqref{eq:cover} and \eqref{eq:bbound} give
		\begin{equation}\label{eq:dplus}
			d_++\sum_{v\notin Y}z_v\ <\ \frac n3-\frac b2+d_-\ \le\ \frac{109}2\delta n
			+7\sqrt\delta\,n\ \le\ 8\sqrt\delta\,n .
		\end{equation}
		Set $D=\{v\in Y:z_v>\tfrac58\}\cup\{v\notin Y:z_v>\tfrac18\}$. Each vertex of the
		first set contributes more than $\tfrac18$ to $d_+$ and each of the second more
		than $\tfrac18$ to the other sum in \eqref{eq:dplus}, and the two sets are
		disjoint, so $|D|\le64\sqrt\delta\,n$. A triple in $V(H)\setminus D$ containing at
		most one vertex of $Y$ has total weight at most
		$\tfrac58+\tfrac18+\tfrac18=\tfrac78$ and hence is not an edge by
		\eqref{eq:cover}. Taking $A=Y$ gives alternative (ii).
		
		Suppose now that $a\ge(\tfrac13-27\delta)n$. By \eqref{eq:J},
		$b<\tfrac{2n}3-2a\le54\delta n$. Let $T=K_X\vee I_{V(G)\setminus X}$ and let $K$ be
		obtained from $G$ by deleting all edges of $G[Y]$. If $b=0$ then $K=G$; otherwise
		the deleted graph has maximum degree at most $b$, so in either case
		$\|A(G)-A(K)\|_{\mathrm{op}}\le b$, and the Rayleigh principle gives
		$\rho(K)\ge\rho(G)-b\ge(\tfrac23-55\delta)n$. The quotient bound with $y=0$ gives
		$\rho(T)/n\le r(x,0)\le r_x\le\tfrac23$, so $D_0:=\rho(T)-\rho(K)\le55\delta n$.
		Since $\delta\le10^{-4}$ we have $\tfrac13-27\delta>\tfrac14$, so $n/4\le a\le n/3$
		and Lemma~\ref{lem:split} bounds the number $M$ of edges of $T$ missing from $K$ by
		\begin{equation}\label{eq:M2}
			M\le416D_0n\le22880\delta n^2\le23000\delta n^2 .
		\end{equation}
		Write $Z_X=\sum_{v\in X}z_v$, $Z_B=\sum_{v\notin X}z_v$ and $S=Z_X+Z_B<n/3$.
		Summing $z_x+z_y$ over all pairs with $x\in X$ and $y\notin X$ gives
		$(n-a)Z_X+aZ_B$; each of the at least $a(n-a)-M$ present cross edges has weight sum
		at least one and every missing cross pair has nonnegative weight sum, so
		$(n-a)Z_X+aZ_B\ge a(n-a)-M$. Putting $d=a-Z_X\ge0$ and substituting
		$Z_X=a-d$, $Z_B=S-a+d$ gives $(n-2a)d\le a(S-a)+M$. As $n-2a\ge n/3$,
		$S-a<n/3-a\le27\delta n$ and $a\le n/3$, \eqref{eq:M2} yields
		\[
		d\ \le\ \frac{(n/3)(27\delta n)+23000\delta n^2}{n/3}\ =\ 69027\delta n ,
		\]
		whence $Z_B=S-a+d\le n/3-a+d\le69054\delta n$. The set
		$B_{\mathrm{bad}}=\{v\notin X:z_v\ge\tfrac13\}$ therefore has at most
		$3Z_B\le207162\delta n$ vertices. Let $A'=X\cup B_{\mathrm{bad}}$. Every vertex
		outside $A'$ has weight less than $\tfrac13$, so no edge is disjoint from $A'$ and
		$A'$ is a vertex cover; and
		$(\tfrac13-27\delta)n\le a\le|A'|\le\tfrac n3+207162\delta n
		\le(\tfrac13+210000\delta)n$, which is alternative (i).
	\end{proof}
	
	\begin{proof}[Proof of Theorem~\ref{thm:fsnon}]
		Choose
		\[
		0<\delta\le\min\Bigl\{10^{-4},\frac\beta{210000},\frac\beta{109},
		\Bigl(\frac\beta{64}\Bigr)^2\Bigr\},\qquad N_0\ge\max\{12,\delta^{-1/2}\},
		\]
		and apply Lemma~\ref{lem:dichot}. In alternative (i) take $i=1$, $A=A'$ and $D=\emptyset$;
		since $A'$ is a vertex cover, every edge lies in $T_1(A')$, and the size estimate
		follows from $210000\delta\le\beta$. In alternative (ii) take $i=2$ and keep the
		sets $A$ and $D$; the inequalities $109\delta\le\beta$ and $64\sqrt\delta\le\beta$
		give the size bounds, and the edge conclusion is $E(F-D)\subseteq T_2(A)$.
	\end{proof}

	\section{From stability to perfect matchings}\label{sec:pm}
	
	\subsection{Absorbing sets}\label{sec:absorb}
	
	\medskip\noindent\emph{The non-partite case.}
	The absorbing lemma of Lin, Lu, Yuan and
	Zhao~\cite[Lemma~2.10]{LLYZ} requires
	$\sigma(H)>(\tfrac23+\gamma)n$ for a fixed $\gamma>0$,
	so it does not apply at the exact threshold.
	We therefore construct the absorbing set directly.
	
	\begin{lem}\label{lem:absnp}
		For every $0<\mu<1$ there are $\eta>0$ and $n_0$ such that the following holds. If
		$H$ is a $3$-graph on $n\ge n_0$ vertices with $\sigma(H)\ge\tfrac{33}{50}n$, then
		there is $A\subseteq V(H)$ with $|A|\le\mu n$ and $3\mid|A|$ such that $H[A\cup W]$
		has a perfect matching for every $W\subseteq V(H)\setminus A$ with $|W|\le\eta n$
		and $3\mid|W|$.
	\end{lem}
	
	Throughout, $H$ is a $3$-graph on $n$ vertices with $\sigma(H)\ge33n/50$ and $n$
	large, and each link is regarded as a graph on $V(H)$ by adding its defining vertex
	as an isolated vertex, which does not change its spectral radius. Put
	$\beta=10^{-6}$ and $c=\tfrac1{200}$, and define an auxiliary graph $R$ on $V(H)$
	by joining distinct $u,v$ when
	\begin{equation}\label{eq:Rnp}
		\bigl|E(L_H(u))\cap E(L_H(v))\bigr|\ \ge\ \beta n^2 .
	\end{equation}
	
	\begin{proof}
		\medskip\noindent\textbf{Claim 1.} For $n$ large, $\delta(R)>cn$, $\alpha(R)\le2$, and $e(H)>n^3/15$.
		
		By Lemma~\ref{lem:basic}(ii), every link has at least $\tfrac12(\tfrac{33}{50})^2n^2$
		edges. If $u,v,w$ were pairwise nonadjacent in $R$, then inclusion and exclusion,
		discarding the nonnegative triple term, would give
		\[
		\bigl|E(L_H(u))\cup E(L_H(v))\cup E(L_H(w))\bigr|\ \ge\
		\Bigl[\tfrac32\bigl(\tfrac{33}{50}\bigr)^2-3\beta\Bigr]n^2
		\ =\ 0.653397\,n^2\ >\ \binom n2 ,
		\]
		which is impossible since the three edge sets consist of pairs from the same
		$n$-element set. Hence $\alpha(R)\le2$. Counting each hyperedge in its three links,
		\begin{equation}\label{eq:eHnp}
			e(H)=\frac13\sum_{v}e(L_H(v))\ \ge\ \frac16\Bigl(\frac{33}{50}\Bigr)^2n^3
			\ =\ 0.0726\,n^3\ >\ \frac{n^3}{15} .
		\end{equation}
		
		For the minimum degree, fix $u$ and suppose $d_R(u)\le cn$. Write $G=L_H(u)$,
		$Y=N_R(u)\cup\{u\}$, $X=V(H)\setminus Y$ and $r=|Y|$, so that $r/n\le c+1/n\le0.006$
		for $n\ge1000$. For $x\in X$ let $b_x$ be the number of edges of $L_H(x)[X]$ with at
		least one endpoint in $N_{G[X]}(x)$. Then
		\begin{equation}\label{eq:bx}
			\sum_{x\in X}b_x\ \le\ 2\sum_{z\in X}\bigl|E(G[X])\cap E(L_H(z))\bigr|\ <\ 2\beta n^3 .
		\end{equation}
		For the first inequality, consider a triple $\{x,y,z\}\in E(H[X])$: the edge $xy$ of
		$G[X]$ lies in $L_H(z)$ and can account for this triple in $b_x$ and in $b_y$;
		summing over the edges of $G$ inside each triple gives the bound even when a triple
		contributes to more than one term. For the strict inequality, every $z\in X$ is
		nonadjacent to $u$ in $R$, so its common link with $u$ has fewer than $\beta n^2$
		edges.
		
		Let $Z=\{x\in X:b_x>10^{-3}n^2\}$, so $|Z|<0.002n$ by \eqref{eq:bx} and $X\setminus Z$
		is nonempty. Fix $x\in X\setminus Z$, delete from $L_H(x)$ all edges meeting $Y$ and
		then all edges counted by $b_x$, at most $rn+0.001n^2$ in total, and let $F_x$ be
		what remains. Every edge of $F_x$ has both ends in $U_x=X\setminus N_{G[X]}(x)$, and
		$x\in U_x$ since $G$ has no loops, so all vertices outside $U_x$ are isolated and
		\[
		\rho(F_x)=\rho(F_x[U_x])\le\Delta(F_x[U_x])\le|U_x|-1\le n-r-d_{G[X]}(x).
		\]
		Lemma~\ref{lem:basic}(iii) now gives
		$\sigma(H)\le n-r-d_{G[X]}(x)+\sqrt{2rn+0.002n^2}$, hence
		$\Delta(G[X\setminus Z])\le n-r-\sigma(H)+\sqrt{2rn+0.002n^2}$. Deleting from $G$ all
		edges meeting $Y\cup Z$ deletes at most $(r+0.002n)n$ edges, so by
		Lemma~\ref{lem:basic}(iii) again,
		\begin{align*}
		\sigma(H)\le\rho(G)&\le\rho(G[X\setminus Z])+\sqrt{2(r+0.002n)n}\\
		&\le n-r-\sigma(H)+\sqrt{2rn+0.002n^2}+\sqrt{2(r+0.002n)n}.
		\end{align*}
		Dividing by $n$,
		\[
		\frac{2\sigma(H)}n\ \le\ 1-\frac rn+\sqrt{2\Bigl(\frac rn+0.001\Bigr)}
		+\sqrt{2\Bigl(\frac rn+0.002\Bigr)} ,
		\]
		whose right side is at most $1+\sqrt{0.014}+\sqrt{0.016}<1.245$, while its left side
		is at least $1.32$. This contradiction proves $\delta(R)>cn$.

		\medskip\noindent\textbf{Claim 2.} Put $\tau=(c/100)^2$ and $\gamma=\tau/4$. For $n$ large, every graph $R$ on $n$
		vertices with $\delta(R)>cn$ and $\alpha(R)\le2$ has a partition $\mathcal P$ into
		one or two parts, each of size at least $cn/3$, such that any two distinct vertices
		$u,v$ in the same part satisfy, for some $\ell\in\{1,2,3\}$,
		\[
		\#\{\text{simple $u$--$v$ paths of length $\ell$ in }R\}\ \ge\ \gamma n^{\ell-1} .
		\]
		
		Suppose first that no partition $V(R)=P\sqcup Q$ satisfies
		\begin{equation}\label{eq:cut}
			|P|,|Q|\ge cn/2,\qquad e_R(P,Q)<\tau n^2 ,
		\end{equation}
		and take $\mathcal P=\{V(R)\}$. Fix distinct $u,v$. If $uv\in E(R)$ the edge itself
		suffices, as $\gamma<1$. Otherwise put $C=N_R(u)\cap N_R(v)$, $P=N_R(u)\setminus N_R(v)$
		and $Q=N_R(v)\setminus N_R(u)$; since $\alpha(R)\le2$, every other vertex is adjacent
		to $u$ or $v$, so $V(R)=P\sqcup Q\sqcup C\sqcup\{u,v\}$. If $|C|\ge\gamma n$ the
		paths $uzv$ with $z\in C$ suffice. Otherwise
		$|P|=d_R(u)-|C|>(c-\gamma)n\ge cn/2$, and the complement of $P$ contains $N_R(v)$ and
		has size more than $cn$; as \eqref{eq:cut} fails for this cut,
		\[
		e_R(P,Q)\ \ge\ \tau n^2-|P||C|-2n\ \ge\ (\tau-\gamma)n^2-2n\ \ge\ \frac\tau2n^2
		\ \ge\ \gamma n^2
		\]
		for $n$ large, the subtracted terms accounting for edges from $P$ to $C$ and to
		$\{u,v\}$. Every edge $pq$ with $p\in P$, $q\in Q$ gives a distinct simple path
		$upqv$.
		
		Now suppose a partition satisfying \eqref{eq:cut} exists, and let $X$ be the set of
		vertices with more than $\sqrt\tau\,n$ neighbors in the opposite part. The sum of
		cross-part degrees is $2e_R(P,Q)<2\tau n^2$, so $|X|<2\sqrt\tau\,n=cn/50$. Put
		$P_0=P\setminus X$ and $Q_0=Q\setminus X$, both of size at least $0.48cn$. They are
		cliques: if $x,y\in P_0$ were nonadjacent then $\alpha(R)\le2$ would give
		$Q_0\subseteq N_R(x)\cup N_R(y)$, whereas $x$ and $y$ have at most $\sqrt\tau\,n$
		neighbors in $Q$ each, so $0.48cn\le2\sqrt\tau\,n=0.02cn$, a contradiction. Each
		$x\in X$ has more than $cn-|X|>0.98cn$ neighbors in $P_0\cup Q_0$, hence at least
		$cn/3$ in one of the two cliques; assign it there. This gives two parts, each
		containing its core clique and so of size at least $0.48cn$.
		
		Two core vertices are adjacent. An assigned vertex and a core vertex are joined by
		at least $cn/3-1$ paths of length two through the core. Two assigned vertices are
		joined by at least $(cn/3)^2-n$ paths of length three, obtained by choosing a
		neighbor of each in the core and using that the core is a clique. For $n$ large both
		bounds exceed $\gamma n$ and $\gamma n^2$ respectively.

		For distinct $u,v\in V(H)$, a \emph{connector} is a set
		$S\subseteq V(H)\setminus\{u,v\}$ such that $H[S\cup\{u\}]$ and $H[S\cup\{v\}]$ both
		have perfect matchings; and for a three-element set $T$, an \emph{absorber} is a set
		$D\subseteq V(H)\setminus T$ such that $H[D]$ and $H[D\cup T]$ both have perfect
		matchings, $T$ not being required to be an edge.
		
		\medskip\noindent\textbf{Claim 3.} There is an absolute constant $\theta_0>0$ such that, for $n$ large, any two distinct
		vertices in the same part of the partition of Claim 2 have at least
		$\theta_0n^8$ connectors of size $8$.
		
		Fix $u,v$ in the same part and $\ell$ as in Claim 2. For each of the
		at least $\gamma n^{\ell-1}$ paths $p_0=u,p_1,\dots,p_\ell=v$, choose successively
		pairs $B_j\in E(L_H(p_{j-1}))\cap E(L_H(p_j))$ for $j\in[\ell]$, avoiding the path
		vertices and the previous choices. At most eight vertices are forbidden at any
		choice and each lies in fewer than $n$ pairs, so by \eqref{eq:Rnp} there are at least
		$\beta n^2-8n\ge\beta n^2/2$ possibilities once $n\ge16/\beta$. Put
		$S_0=\{p_1,\dots,p_{\ell-1}\}\cup B_1\cup\dots\cup B_\ell$, of size $3\ell-1$. The
		families $\{\{p_{j-1}\}\cup B_j\}$ and $\{\{p_j\}\cup B_j\}$ are perfect matchings of
		$S_0\cup\{u\}$ and $S_0\cup\{v\}$. Add $3-\ell$ disjoint hyperedges avoiding
		$S_0\cup\{u,v\}$ and one another; at most ten vertices are forbidden and at most
		$10\binom{n-1}2\le5n^2$ hyperedges meet them, so by \eqref{eq:eHnp} each choice has at
		least $n^3/15-5n^2\ge n^3/30$ possibilities. The resulting set has size
		$(3\ell-1)+3(3-\ell)=8$ and is a connector, the added edges being appendable to either
		matching. There are at least $\gamma(\beta/2)^\ell(1/30)^{3-\ell}n^8$ constructions,
		and listing the internal path vertices in path order, then $B_1,\dots,B_\ell$ with
		each pair increasing, then the added edges in order of selection with each increasing,
		shows that each $8$-set is counted at most $8!$ times. So
		$\theta_0=\tfrac\gamma{8!}\min_{1\le\ell\le3}(\beta/2)^\ell(1/30)^{3-\ell}$ works.

		\medskip\noindent\textbf{Claim 4.} There is an absolute constant $\zeta>0$ such that, for $n$ large, every pair of
		distinct vertices of $H$ has at least $\zeta n^{38}$ connectors of size $38$.
		
		For two vertices in the same part of $\mathcal P$, extend an eight-vertex connector by
		ten disjoint hyperedges, each avoiding the connector, the two defining vertices and
		the previous additions. At most $40$ vertices are forbidden, so \eqref{eq:eHnp} gives
		at least $n^3/15-40\binom{n-1}2\ge n^3/30$ choices for $n$ large. The extended set has
		size $38$ and is again a connector, and its multiplicity is at most $38!$, so the pair
		has at least
		\begin{equation}\label{eq:38same}
			\frac{\theta_0}{30^{10}\,38!}\,n^{38}
		\end{equation}
		connectors. This settles the claim if $\mathcal P$ has one part.
		
		Suppose $\mathcal P=\{P,Q\}$ with $|P|,|Q|\ge cn/3$, and put $\xi=c/10^6$. An edge has
		\emph{type} $j$ if it contains exactly $j$ vertices of $P$, and a type is \emph{rich}
		if it contains at least $\xi n^3$ edges. We claim two consecutive types are rich. If
		not, the rich types lie in one of $\{0,2\}$, $\{1,3\}$, $\{0,3\}$, these being the
		maximal subsets of $\{0,1,2,3\}$ without consecutive integers. Retain only such a set
		and call the result $H'$; fewer than $4\xi n^3$ edges are deleted, so the degree losses
		sum to at most $12\xi n^3$. If the retained types are $\{0,2\}$, the link in $H'$ of a
		vertex of $P$ is bipartite between $P$ and $Q$; if they are $\{1,3\}$, the same holds
		for a vertex of $Q$; if they are $\{0,3\}$, the link of a vertex in the smaller of
		$P,Q$ has all its nonisolated vertices in one part, so its spectral radius is at most
		$n/2$. A bipartite graph with part sizes $a,b$ has spectral radius at most
		$\sqrt{ab}\le(a+b)/2\le n/2$. In each case the part under consideration has at least
		$cn/3$ vertices, so one of them, say $v$, has degree loss at most
		$12\xi n^3/(cn/3)=36\xi n^2/c$, and Lemma~\ref{lem:basic}(iii) gives
		\[
		\rho(L_H(v))\ \le\ \rho(L_{H'}(v))+\sqrt{2\bigl(d_H(v)-d_{H'}(v)\bigr)}
		\ \le\ \Bigl(\frac12+\sqrt{\frac{72\xi}c}\Bigr)n\ <\ 0.509\,n,
		\]
		contradicting $\sigma(H)\ge0.66n$.
		
		Choose consecutive rich types $i,i+1$ with $0\le i\le2$, fix $u\in P$ and $v\in Q$, and
		choose disjoint edges $e=\{b,x_1,x_2\}$ of type $i$ and $f=\{a,y_1,y_2\}$ of type
		$i+1$ avoiding $u,v$, with $b\in Q$ and $a\in P$; such $b$ and $a$ exist because an
		edge of type $i\le2$ meets $Q$ and an edge of type $i+1\ge1$ meets $P$. The pairs remaining after deleting $b$ and $a$ each contain exactly $i$
		vertices of $P$, so ordering each pair first by part and then by a fixed vertex order
		puts $x_j$ and $y_j$ in the same part for $j=1,2$. At most $5\binom{n-1}2$ edges are
		excluded at either choice, so richness leaves at least $\xi n^3/2$ choices at each
		stage. By Claim 3 choose four pairwise disjoint eight-vertex connectors
		avoiding $u,v$ and $e\cup f$: $C_0$ for $u,a$; $C_3$ for $b,v$; and $C_j$ for
		$x_j,y_j$. At most $32$ vertices are forbidden before any choice and at most
		$32\binom{n-1}7\le32n^7$ eight-element sets meet them, so each choice has at least
		$\theta_0n^8/2$ possibilities. Put $S=e\cup f\cup C_0\cup C_1\cup C_2\cup C_3$, of size
		$38$. The edge $f$ together with perfect matchings on $C_0\cup\{u\}$, $C_3\cup\{b\}$,
		$C_1\cup\{x_1\}$, $C_2\cup\{x_2\}$ covers $S\cup\{u\}$, and the edge $e$ together with
		perfect matchings on $C_0\cup\{a\}$, $C_3\cup\{v\}$, $C_1\cup\{y_1\}$, $C_2\cup\{y_2\}$
		covers $S\cup\{v\}$. There are at least $(\xi n^3/2)^2(\theta_0n^8/2)^4=\xi^2\theta_0^4
		n^{38}/64$ constructions, each set arising at most $38!$ times, so with
		\eqref{eq:38same} one may take
		$\zeta=\min\{\theta_0/(30^{10}38!),\ \xi^2\theta_0^4/(64\cdot38!)\}$.

		\medskip\noindent\textbf{Claim 5.} There is an absolute constant $\omega>0$ such that, for $n$ large, every
		three-element set has at least $\omega n^{78}$ absorbers of size $78$.
		
		Fix $T=\{t_1,t_2,t_3\}$ and choose $xy\in E(L_H(t_1))$ avoiding $T$; the link has at
		least $\tfrac12(33/50)^2n^2$ edges and avoiding $t_2,t_3$ excludes at most $2n$ of
		them, so there are at least $n^2/10$ choices. By Claim 4 choose
		disjoint $38$-vertex connectors $C_2$ for $t_2,x$ and $C_3$ for $t_3,y$, both avoiding
		$T\cup\{x,y\}$; at most $43$ vertices are forbidden and at most $43n^{37}$ sets of size
		$38$ meet them, leaving at least $\zeta n^{38}/2$ choices at each stage. Then
		$D=\{x,y\}\cup C_2\cup C_3$ has size $78$, perfect matchings on $C_2\cup\{x\}$ and
		$C_3\cup\{y\}$ cover $D$, and the edge $\{t_1,x,y\}$ with perfect matchings on
		$C_2\cup\{t_2\}$ and $C_3\cup\{t_3\}$ covers $D\cup T$. There are at least
		$(n^2/10)(\zeta n^{38}/2)^2=\zeta^2n^{78}/40$ constructions and each absorber occurs at
		most $78!$ times, so $\omega=\zeta^2/(40\cdot78!)$ works.

		\medskip\noindent\emph{Completion of the proof.} Let $q=78$ and choose
		$0<\theta\le\min\{\mu/(2q),\ \omega/(4q^2)\}$ and $\eta=\omega\theta/8$. Select each
		$q$-element subset of $V(H)$ independently with probability $p=\theta n/\binom nq$ and
		let $\mathcal F$ be the resulting family, so $\mathbb E|\mathcal F|=\theta n$. For a
		fixed $T$, Claim 5 gives at least $\omega\binom nq$ absorbers, so the
		number $Y_T$ of selected ones has $\mathbb EY_T\ge\omega\theta n$. By the Chernoff
		bounds, $\Pr(|\mathcal F|>2\theta n)\le e^{-\theta n/3}$ and
		$\Pr(Y_T<\omega\theta n/2)\le e^{-\omega\theta n/8}$, and a union bound over the at
		most $n^3$ sets $T$ shows that the probability that some bound fails tends to zero.
		Let $Z$ be the number of unordered pairs of distinct intersecting members of
		$\mathcal F$; a fixed $q$-set meets at most $q\binom{n-1}{q-1}=\tfrac{q^2}n\binom nq$
		others, so $\mathbb EZ\le\tfrac12\binom nq\tfrac{q^2}n\binom nq p^2=\tfrac{q^2\theta^2}2n$
		and $\Pr(Z>q^2\theta^2n)\le\tfrac12$ by Markov. Hence for $n$ large there is a family
		with $|\mathcal F|\le2\theta n$, $Y_T\ge\omega\theta n/2$ for every $T$, and
		$Z\le q^2\theta^2n$.
		
		Delete one member of each intersecting pair, which removes at most $Z$ sets, and then
		delete every remaining $D$ for which $H[D]$ has no perfect matching; the latter removes
		no absorber. In the resulting family $\mathcal F'$ every $T$ still has at least
		$\omega\theta n/2-q^2\theta^2n\ge\omega\theta n/4$ absorbers. Put
		$A=\bigcup_{D\in\mathcal F'}D$; the members are disjoint and each has a perfect
		matching, so $H[A]$ does, and $|A|=q|\mathcal F'|\le2q\theta n\le\mu n$ with
		$3\mid|A|$ since $3\mid q$. Finally, given $W$ with $|W|\le\eta n$ and $3\mid|W|$,
		partition it into triples $T_1,\dots,T_s$ with $s=|W|/3\le\omega\theta n/24$ and assign
		to each a distinct absorber $D_i\in\mathcal F'$; when $T_i$ is considered fewer than
		$\omega\theta n/24$ members have been used against at least $\omega\theta n/4$
		available. Perfect matchings on each $D_i\cup T_i$, together with perfect matchings on
		the unassigned members, cover $A\cup W$.
	\end{proof}

	\medskip\noindent\emph{The partite case.}
	Here the auxiliary graph is built from a rigidity statement for pairs of links that
	are almost complementary, which also yields expansion and so replaces the separate
	independence and degree estimates of the non-partite argument. Write $c=1/\sqrt2$.
	
	\begin{lem}\label{lem:rigid}
		For every $\theta>0$ there are $\omega>0$ and $Q$ with the following property. Let
		$G$ and $F$ be bipartite graphs on common classes $Y,Z$ of size $q\ge Q$ with
		$\rho(G),\rho(F)\ge(c-\omega)q$ and $|E(G)\cap E(F)|\le\omega q^2$. Then, after
		possibly interchanging $Y$ and $Z$, there is $S\subseteq Y$ with
		$\bigl||S|-q/2\bigr|\le\theta q$,
		$|E(G)\triangle(S\times Z)|\le\theta q^2$ and
		$|E(F)\triangle((Y\setminus S)\times Z)|\le\theta q^2$.
	\end{lem}
	
	\begin{proof}
		For a zero-one matrix $B$ put $E=\|B\|_F^2$, $R=E-s_1(B)^2$, and let $M(B)$ be the
		sum of the squares of the $2\times2$ minors of $B$, each counted once. By
		Cauchy--Binet applied to the second elementary symmetric function of $B^{\top}B$,
		$M(B)=\sum_{i<j}s_i(B)^2s_j(B)^2$, and separating the terms containing $s_1(B)^2$,
		whose sum is $s_1(B)^2R$, from the rest, whose sum is at most $R^2/2$, gives
		$M(B)\le ER$. For a position $(i,j)$ with $B_{ij}=1$, the rectangle
		$P_{ij}=\{k:B_{kj}=1\}\times\{\ell:B_{i\ell}=1\}$ differs from $B$ in
		$D_{ij}=\sum_{k,\ell}(B_{k\ell}-B_{kj}B_{i\ell})^2$ entries, each summand being the
		square of the minor on rows $i,k$ and columns $j,\ell$; summing over these positions
		counts each unordered minor at most four times, so
		$\sum_{(i,j):B_{ij}=1}D_{ij}\le4M(B)$ and, when $E>0$, some rectangle differs from
		$B$ in at most $4M(B)/E\le4R$ entries.
		
		Suppose the lemma fails for some $\theta$, and take counterexamples with
		$q_j\ge j$ and $\omega_j\le1/j$; all asymptotics refer to this sequence. By
		Lemma~\ref{lem:basic}(ii), $e(G),e(F)\ge(\tfrac12-o(1))q^2$, while
		$e(G)+e(F)\le q^2+|E(G)\cap E(F)|$, so
		\[
		e(G)=e(F)=\bigl(\tfrac12+o(1)\bigr)q^2,\qquad
		e(G)-\rho(G)^2=o(q^2),\qquad e(F)-\rho(F)^2=o(q^2).
		\]
		The matrix observation gives a rectangle $A\times D$ with $A\subseteq Y$,
		$D\subseteq Z$ differing from $G$ in $o(q^2)$ edges; put $a=|A|/q$ and $b=|D|/q$, so
		$ab=\tfrac12+o(1)$. Also
		$|E(F)\triangle E(G^c)|=q^2-e(G)-e(F)+2|E(G)\cap E(F)|=o(q^2)$, the complement being
		taken in $Y\times Z$, so $F$ differs in $o(q^2)$ entries from the complement $P$ of
		$A\times D$. Changing one entry of a zero-one matrix changes at most $(q-1)^2$
		squared minors by at most one each, so $M(P)=M(F)+o(q^4)=o(q^4)$. A minor of $P$ is
		nonzero exactly when its rows lie on opposite sides of $A$ and its columns on
		opposite sides of $D$, whence $M(P)=a(1-a)b(1-b)q^4$ and $a(1-a)b(1-b)=o(1)$. Since
		$ab=\tfrac12+o(1)$, both $a$ and $b$ are bounded away from zero, so along a
		subsequence either $a\to\tfrac12$ and $b\to1$, or $a\to1$ and $b\to\tfrac12$. In the
		first case $G$ differs in $o(q^2)$ edges from $A\times Z$ and $F$ from
		$(Y\setminus A)\times Z$; in the second the same holds after interchanging the
		classes. For large $j$ this contradicts the choice of counterexamples.
	\end{proof}
	
	Fix $\theta=10^{-4}$, let $\omega$ and $Q$ be given by Lemma~\ref{lem:rigid}, and put
	\[
	\delta_*=\beta=\min\{\omega,\theta\},\qquad \kappa=\theta^2,\qquad d_*=0.49 .
	\]
	For a $q$-balanced $3$-partite $H$ with classes $X,Y,Z$, let $R_X$ join distinct
	$x,x'\in X$ when $|E(L_H(x))\cap E(L_H(x'))|\ge\beta q^2$, and define $R_Y$, $R_Z$
	analogously.
	
	\begin{lem}\label{lem:expand}
		For $q$ large, if $\sigma(H)\ge(c-\delta_*)q$, then each of $R_X,R_Y,R_Z$ has
		minimum degree at least $0.49q$, and $e_R(A,B)\ge\kappa q^2$ whenever $A,B$ are
		disjoint sets of at least $0.4q$ vertices.
	\end{lem}
	
	\begin{proof}
		By Lemma~\ref{lem:basic}(ii) every vertex satisfies
		\begin{equation}\label{eq:degstar}
			d_H(v)\ge(c-\delta_*)^2q^2\ge\bigl(\tfrac12-2\delta_*\bigr)q^2\ge d_*q^2 ,
		\end{equation}
		and by symmetry we treat $R_X$. Fix $u\in X$ and let $B$ be the set of vertices of
		$X\setminus\{u\}$ nonadjacent to $u$ in $R_X$; we may assume $B\ne\emptyset$ and
		pick $w_0\in B$. Lemma~\ref{lem:rigid} applied to $L_H(u)$ and $L_H(w_0)$ gives,
		after possibly interchanging $Y$ and $Z$, a set $S\subseteq Y$ with
		$|S|\ge(\tfrac12-\theta)q$ and $|E(L_H(u))\triangle(S\times Z)|\le\theta q^2$. For
		every $w\in B$,
		$e(L_H(w)[S,Z])\le|E(L_H(w))\cap E(L_H(u))|+|(S\times Z)\setminus E(L_H(u))|
		\le(\beta+\theta)q^2$; summing over $w$ and averaging over $S$ produces $y\in S$
		with $d_H(y;B,Z)\le|B|(\beta+\theta)q^2/|S|\le3(\beta+\theta)q^2$. Since at most
		$(q-|B|)q$ edges through $y$ have their $X$-vertex outside $B$, \eqref{eq:degstar}
		gives $(\tfrac12-2\delta_*)q^2\le(q-|B|)q+3(\beta+\theta)q^2$, so
		$d_{R_X}(u)=q-1-|B|\ge(\tfrac12-8\theta)q-1\ge0.49q$ for $q$ large.
		
		For the second assertion, suppose disjoint $A,B\subseteq X$ of size at least $0.4q$
		have $e_{R_X}(A,B)<\kappa q^2$. Fewer than $\theta q$ vertices of either set have
		more than $\theta q$ neighbors in the other, so we may pick nonadjacent $a_0\in A$
		and $b_0\in B$ each with at most $\theta q$ neighbors opposite. Lemma~\ref{lem:rigid}
		applied to their links gives, after possibly interchanging $Y,Z$, a set
		$S\subseteq Y$ with $s=|S|/q$ satisfying $|s-\tfrac12|\le\theta$,
		$|E(L_H(a_0))\triangle(S\times Z)|\le\theta q^2$ and
		$|E(L_H(b_0))\triangle((Y\setminus S)\times Z)|\le\theta q^2$; consequently
		$e(L_H(a)[Y\setminus S,Z])\le(\beta+\theta)q^2$ for $a\in A\setminus N_{R_X}(b_0)$
		and $e(L_H(b)[S,Z])\le(\beta+\theta)q^2$ for $b\in B\setminus N_{R_X}(a_0)$. Let
		$\mathcal T$ consist of the legal triples in
		$(A\times S\times Z)\cup(B\times(Y\setminus S)\times Z)\cup((X\setminus(A\cup B))\times Y\times Z)$;
		the bounds just given, together with the at most $\theta q$ exceptional vertices in
		each of $A$ and $B$, yield
		$|E(H)\setminus\mathcal T|\le(|A|+|B|)(\beta+\theta)q^2+2\theta q^3\le6\theta q^3$,
		so some $z\in Z$ has at most $6\theta q^2$ link edges outside the bipartite template
		$T=L_{\mathcal T}(z)$. Writing $a=|A|/q$ and $b=|B|/q$ and grouping the columns by
		$S$ and $Y\setminus S$, the nonzero eigenvalues of $q^{-2}T^{\top}T$ are those of
		\[
		\begin{pmatrix}
			s(1-b) & (1-a-b)\sqrt{s(1-s)}\\
			(1-a-b)\sqrt{s(1-s)} & (1-s)(1-a)
		\end{pmatrix}.
		\]
		Since $a,b\ge0.4$, each diagonal entry is at most $0.3+0.6\theta$ and the
		off-diagonal entry at most $0.1$, so the largest eigenvalue is at most the largest
		row sum, namely $0.4+0.6\theta<0.64^2$, giving $\rho(T)<0.64q$. As $L_H(z)$ lies in
		$T$ together with at most $6\theta q^2$ further edges,
		Lemma~\ref{lem:basic}(ii) and the triangle inequality give
		$\rho(L_H(z))<(0.64+\sqrt{6\theta})q<0.665q$, contradicting
		$\rho(L_H(z))\ge(c-\delta_*)q>0.7q$.
	\end{proof}
	
	For distinct $u,v$ in the same class, an \emph{$8$-connector} is an $8$-element set
	$S\subseteq V(H)\setminus\{u,v\}$ such that $H[S\cup\{u\}]$ and $H[S\cup\{v\}]$ both
	have perfect matchings; such a set has two vertices in the class of $u,v$ and three
	in each other class. A \emph{balanced triple} is a set with one vertex in each
	class, not required to be an edge, and an \emph{absorber} for it is an $18$-element
	set $A$ disjoint from it such that $H[A]$ and $H[A\cup T]$ both have perfect
	matchings.
	
	\begin{lem}\label{lem:abs}
		For every $\alpha>0$ there are $\mu>0$ and $Q_\alpha$ such that, if $q\ge Q_\alpha$
		and $\sigma(H)\ge(c-\delta_*)q$, then $H$ has a matching $M_{\mathrm{abs}}$ with at
		most $\alpha q$ edges such that $H[V(M_{\mathrm{abs}})\cup W]$ has a perfect matching
		for every $W\subseteq V(H)\setminus V(M_{\mathrm{abs}})$ with
		$|W\cap X|=|W\cap Y|=|W\cap Z|\le\mu q$.
	\end{lem}
	
	\begin{proof}
		\emph{Connectors.} We first show that for an absolute $\gamma>0$ every pair of
		distinct vertices in the same class has at least $\gamma q^8$ connectors of size
		eight. Take $u,v\in X$ and $R=R_X$. If $|N_R(u)\cap N_R(v)|\ge q/100$ there are at
		least $q/100$ paths of length two from $u$ to $v$; otherwise
		$A=N_R(u)\setminus(N_R(v)\cup\{v\})$ and $B=N_R(v)\setminus(N_R(u)\cup\{u\})$ are
		disjoint of size at least $0.49q-q/100-1\ge0.4q$, and Lemma~\ref{lem:expand} gives
		at least $\kappa q^2$ edges between them and hence that many paths of length three.
		Along such a path $x_0=u,x_1,\dots,x_\ell=v$ with $\ell\in\{2,3\}$, choose pairwise
		disjoint pairs $P_j\in E(L_H(x_{j-1}))\cap E(L_H(x_j))$; at most four vertices are
		forbidden at any step and each lies in at most $q$ pairs, so at least $\beta q^2/2$
		choices remain. The internal path vertices together with the pairs form a set of
		$3\ell-1$ vertices carrying the matchings $\{\{x_{j-1}\}\cup P_j\}$ on its union with
		$u$ and $\{\{x_j\}\cup P_j\}$ on its union with $v$. For $\ell=3$ this is an
		$8$-connector; for $\ell=2$ add any edge disjoint from the five vertices and from
		$u,v$, of which there are at least $d_*q^3/2$ by \eqref{eq:degstar} once $q$ is
		large. A fixed $8$-set arises at most $8!$ times, since the construction is encoded
		by an ordering of its vertices, so
		$\gamma=\tfrac1{8!}\min\{\tfrac1{100},\kappa\}\bigl(\tfrac12\min\{\beta,d_*\}\bigr)^3$
		works.
		
		\emph{Absorbers.} Fix a balanced triple $T=\{x_1,x_2,x_3\}$ and choose an edge
		$\{t_1,t_2,x_3\}$ with $t_1\ne x_1$ and $t_2\ne x_2$, of which there are at least
		$d_*q^2/2$; then choose disjoint $8$-connectors $S_1$ for $x_1,t_1$ and $S_2$ for
		$x_2,t_2$, avoiding $T\cup\{t_1,t_2\}$ and each other. At most a bounded number of
		vertices are forbidden and each lies in at most $(3q)^7$ eight-sets, so at least
		$\gamma q^8/2$ choices remain at each step. Put $A=S_1\cup S_2\cup\{t_1,t_2\}$; the
		matchings on $S_1\cup\{t_1\}$ and $S_2\cup\{t_2\}$ cover $A$, while those on
		$S_1\cup\{x_1\}$ and $S_2\cup\{x_2\}$ together with $\{t_1,t_2,x_3\}$ cover
		$A\cup T$. A fixed $18$-set arises at most $18!$ times, so every balanced triple has
		at least $a_0q^{18}$ absorbers with $a_0=d_*\gamma^2/(8\cdot18!)$.
		
		\emph{Selection.} Let $\mathcal F$ be the family of balanced $18$-sets $A$ for which
		$H[A]$ has a perfect matching, so $|\mathcal F|\le q^{18}$ and a fixed vertex lies in
		at most $q^{17}$ members, whence at most $9q^{35}$ unordered pairs of distinct
		intersecting members. Put
		$\vartheta=\min\{\tfrac12,a_0/144,\alpha/24\}$, $p=\vartheta q^{-17}$ and
		$\mu=a_0\vartheta/8$, and select each member of $\mathcal F$ independently with
		probability $p$, and let $\mathcal F_0$ be the selected family. The expected size is at most $\vartheta q$ and the expected number
		of intersecting pairs at most $9\vartheta^2q$, so by Markov each of the events
		$|\mathcal F_0|>4\vartheta q$ and $\#\{\text{intersecting pairs}\}>36\vartheta^2q$
		has probability at most $\tfrac14$; and for a fixed balanced triple the number
		$N_T$ of its selected absorbers has mean at least $a_0\vartheta q$, so
		$\Pr(N_T<a_0\vartheta q/2)\le\exp(-a_0\vartheta q/8)$ and a union bound over the
		$q^3$ triples leaves $o(1)$. Fix an outcome satisfying all these bounds, delete one
		member of each intersecting pair, and let $\mathcal F_1$ be the resulting pairwise
		disjoint family; every balanced triple retains at least
		$a_0\vartheta q/2-36\vartheta^2q\ge a_0\vartheta q/4=2\mu q$ absorbers, by
		$\vartheta\le a_0/144$. Let $M_{\mathrm{abs}}$ be the union of a perfect matching on
		each member, so $|M_{\mathrm{abs}}|=6|\mathcal F_1|\le24\vartheta q\le\alpha q$.
		
		Finally, given $W$ with $w=|W\cap X|\le\mu q$, partition it into balanced triples
		$T_1,\dots,T_w$ and assign to each an unused absorber in $\mathcal F_1$, which is
		possible since each triple has at least $2\mu q$ available and fewer than $w\le\mu q$
		have been used; every chosen absorber avoids $W$ because $W$ avoids
		$V(M_{\mathrm{abs}})$. Replacing the matching on each chosen $A_i$ by one on
		$A_i\cup T_i$ covers $V(M_{\mathrm{abs}})\cup W$.
	\end{proof}

	\subsection{The nonextremal case}\label{sec:nonextremal}
	
	\medskip\noindent\emph{The non-partite case.}
	For $n\in3\mathbb N$ and $\varepsilon>0$, call an $n$-vertex $3$-graph $H$
	\emph{$\varepsilon$-nonextremal} if
	\begin{equation}\label{eq:nonext}
		|E(H)\setminus T_i(A)|\ \ge\ \varepsilon n^3\qquad
		\text{for every }i\in\{1,2\}\text{ and every }A\subseteq V(H)\text{ with }|A|=in/3 .
	\end{equation}
	The distance counts edges outside a space barrier; it does not require the edges
	inside the barrier to be present. The stability theorem turns this into a robustness
	statement.
	
	\begin{lem}\label{lem:robustnp}
		For every $\varepsilon>0$ there are $\eta>0$ and $n_0$ such that the following
		holds. Let $n\ge n_0$ be divisible by $3$ and let $H$ be an $\varepsilon$-nonextremal
		$n$-vertex $3$-graph with $\sigma(H)>\tfrac{2n}3-2$. If $R\subseteq V(H)$ and $J$ is
		a subhypergraph of $H-R$ with $|R|\le\eta n$ and $\Delta_1(J)\le\eta n^2$, then
		$(H-R)-E(J)$ has a perfect fractional matching.
	\end{lem}
	
	\begin{proof}
		Since an $\varepsilon$-nonextremal hypergraph is $\varepsilon'$-nonextremal for
		$\varepsilon'\le\varepsilon$, we may assume $\varepsilon\le\tfrac1{100}$. Put
		$\beta=\varepsilon/10$ and let $\delta$ be given by Theorem~\ref{thm:fsnon} for this
		$\beta$; decreasing $\delta$ is harmless, so assume also $\delta<\tfrac16$. Choose
		\begin{equation}\label{eq:etachoice}
			0<\eta\le\min\{\varepsilon/10,\tfrac18\},\qquad 4\eta+\sqrt{2\eta}<\delta/2 ,
		\end{equation}
		and take $n_0$ large in terms of these constants, with $n_0\ge12$, $2/n_0\le\delta/2$
		and $(1-\eta)n_0$ above the order threshold of Theorem~\ref{thm:fsnon}.
		
		Let $F=(H-R)-E(J)$ and $N=n-|R|$. For $v\in V(F)$ the link $L_H(v)$ has at most $n$
		vertices and $\rho(L_H(v))>\tfrac{2n}3-2\ge\tfrac n2$, so deleting $R$ costs at most
		$4|R|$ by the second case of Lemma~\ref{lem:basic}(iv), and deleting the at most $\eta n^2$ link edges
		contributed by $J$ costs at most $\sqrt{2\eta}\,n$ by Lemma~\ref{lem:basic}(iii). Hence
		\begin{equation}\label{eq:sigmaF}
			\sigma(F)>\frac{2n}3-2-4\eta n-\sqrt{2\eta}\,n\ \ge\ \Bigl(\frac23-\delta\Bigr)n
			\ \ge\ \Bigl(\frac23-\delta\Bigr)N ,
		\end{equation}
		the first inequality by \eqref{eq:etachoice} and $2/n\le\delta/2$, the second by
		$N\le n$.
		
		Suppose $F$ has no perfect fractional matching. Theorem~\ref{thm:fsnon} gives
		$i\in\{1,2\}$ and $A,D\subseteq V(F)$ with
		\begin{equation}\label{eq:ADnp}
			\bigl||A|-iN/3\bigr|\le\beta N,\qquad|D|\le\beta N,\qquad E(F-D)\subseteq T_i(A).
		\end{equation}
		Enlarge or reduce $A$ to a set $A^*\subseteq V(H)$ of size $in/3$, changing exactly
		$\bigl||A|-in/3\bigr|$ vertices, so that
		\begin{equation}\label{eq:Astar}
			|A\,\triangle\,A^*|=\bigl||A|-in/3\bigr|\le\bigl||A|-iN/3\bigr|+i(n-N)/3\le(\beta+\eta)n .
		\end{equation}
		Let $e\in E(H)\setminus T_i(A^*)$. If $e\notin E(J)$ and $e$ avoids
		$R\cup D\cup(A\,\triangle\,A^*)$, then $e\in E(F-D)$, and every vertex of $e$ has the
		same membership in $A$ as in $A^*$, so $|e\cap A|=|e\cap A^*|<i$, contradicting
		\eqref{eq:ADnp}. Every such edge therefore lies in $E(J)$ or meets
		$R\cup D\cup(A\,\triangle\,A^*)$. A fixed vertex lies in at most $\binom{n-1}2$
		triples, and counting incidences gives
		$3|E(J)|=\sum_{v}d_J(v)\le N\eta n^2$, so by \eqref{eq:ADnp} and \eqref{eq:Astar},
		\begin{align*}
			|E(H)\setminus T_i(A^*)|
			&\le \bigl(|R|+|D|+|A\triangle A^*|\bigr)
			\binom{n-1}{2}+|E(J)|\\
			&\le (2\beta+2\eta)n\cdot\frac{n^2}{2}
			+\frac{\eta n^3}{3}\\
			&=\left(\beta+\frac{4\eta}{3}\right)n^3
			\le \frac{7\varepsilon}{30}n^3
			<\varepsilon n^3.
		\end{align*}
		As $|A^*|=in/3$, this contradicts \eqref{eq:nonext}.
	\end{proof}
	
	For a nonnegative function $f$ on the edges of a hypergraph write
	$\mathrm{supp}(f)=\{e:f(e)>0\}$. The next lemma turns robustness into an almost
	perfect matching, following the packing and rounding argument
of~\cite[Section~4]{LY}; it is stated so that it applies verbatim in both settings, the
	partite version being obtained by reading ``uncovered vertices'' class by class.
	
	\begin{lem}\label{lem:almost}
		For every $\eta>0$ and $\zeta>0$ there is $N_0$ such that the following holds. Let
		$F$ be an $N$-vertex $3$-graph with $N\ge N_0$ such that $F-E(J)$ has a perfect
		fractional matching whenever $J\subseteq F$ and $\Delta_1(J)\le\eta N^2$. Then $F$
		has a matching leaving at most $\zeta N$ vertices uncovered.
	\end{lem}
	
	\begin{proof}
		\emph{Small support.} Every hypergraph with a perfect fractional matching has one
		whose support has at most as many edges as the hypergraph has vertices. Indeed, let
		$B$ be the vertex--edge incidence matrix; the set $P=\{f\ge0:Bf=\mathbf1\}$ is
		nonempty and bounded, since $f(e)\le1$ on it, so it has an extreme point $f$. If
		the columns of $B$ indexed by $\mathrm{supp}(f)$ were dependent, there would be
		$d\ne0$ vanishing off $\mathrm{supp}(f)$ with $Bd=0$, and $f\pm td$ would lie in
		$P$ for small $t>0$.
		
		\emph{Packing.} Put $r=\lfloor N/\log N\rfloor$ and take $N_0$ so large that
		$\log N\ge3/\eta$ for $N\ge N_0$. We construct perfect fractional matchings
		$f_1,\dots,f_r$ with pairwise disjoint supports, each of size at most $N$, such that
		\begin{equation}\label{eq:load}
			\ell(xy):=\sum_{j\le r}\ \sum_{e\supseteq\{x,y\}}f_j(e)\ \le\ 3
			\qquad\text{for all distinct }x,y .
		\end{equation}
		Suppose $f_1,\dots,f_s$ have been found with $s<r$ and the partial loads
		$\ell_s(xy)$ at most $3$, and put $Q_s=\{xy:\ell_s(xy)>2\}$. Since
		$\sum_{y\ne x}\ell_s(xy)=2s$ and every edge of $Q_s$ at $x$ carries load more than
		two, $d_{Q_s}(x)\le s$ and $|E(Q_s)|\le sN/2$. Let $J_s$ consist of the edges of $F$
		containing an edge of $Q_s$ together with those in the earlier supports; an edge at
		$v$ containing a pair of $Q_s$ at $v$ is determined by that pair and one further
		vertex, one containing a pair of $Q_s$ away from $v$ by that pair, and the earlier
		supports have at most $sN$ edges, so
		\[
		d_{J_s}(v)\le d_{Q_s}(v)N+|E(Q_s)|+sN\le\tfrac52sN\le3rN\le\frac{3N^2}{\log N}\le\eta N^2 .
		\]
		By hypothesis $F-E(J_s)$ has a perfect fractional matching, and by the first
		paragraph one may take its support to have at most $N$ edges; extended by zero it is
		disjoint from the earlier supports. If $\ell_s(xy)>2$ then every edge containing
		$\{x,y\}$ was deleted and the load is unchanged, and otherwise the added load is at
		most $\sum_{e\ni x}f_{s+1}(e)=1$, so \eqref{eq:load} persists.
		
		\emph{Rounding.} Put $h=\sum_jf_j$, so that $0\le h\le1$, $\sum_{e\ni v}h(e)=r$ for
		every $v$, and $\sum_{e\supseteq\{x,y\}}h(e)\le3$ for every pair. Retain each edge
		$e$ independently with probability $h(e)$, giving $F_0$. Each $d_{F_0}(v)$ is a sum
		of independent Bernoulli variables of mean $r$, so the multiplicative Chernoff bound
		with deviation $1/\log N$ and a union bound over the $N$ vertices leave failure
		probability $2N\exp(-r/(3\log^2N))=o(1)$, because $r\ge N/(2\log N)$ for $N$ large.
		For a fixed pair, $d_{F_0}(xy)$ is a sum of independent Bernoulli variables of mean
		at most three; with $t_N=\lceil\sqrt N\rceil$, expanding
		$\mathbb E\binom{d_{F_0}(xy)}{t_N}$ and keeping the terms with distinct indices gives
		$\Pr(d_{F_0}(xy)\ge t_N)\le(3e/t_N)^{t_N}$, and
		$N^2(3e/t_N)^{t_N}=o(1)$ since $2\log N+t_N\log(3e/t_N)\to-\infty$. So for $N$ large
		some outcome has
		\[
		\Bigl(1-\tfrac1{\log N}\Bigr)r\le d_{F_0}(v)\le\Bigl(1+\tfrac1{\log N}\Bigr)r,
		\qquad \Delta_2(F_0)<t_N .
		\]
		Since $r\to\infty$, $1/\log N\to0$ and $t_N/r\to0$, enlarging $N_0$ makes these
		bounds meet the hypotheses of the Frankl--R\"odl theorem with $k=3$, $D=r$ and the
		given $\zeta$, which supplies the required matching.
	\end{proof}
	
	\begin{thm}\label{thm:nonextnp}
		For every $\varepsilon>0$ there is $n_0$ such that every $\varepsilon$-nonextremal
		$n$-vertex $3$-graph $H$ with $n\ge n_0$, $3\mid n$ and $\sigma(H)>\tfrac{2n}3-2$ has
		a perfect matching.
	\end{thm}
	
	\begin{proof}
		Let $\eta$ be given by Lemma~\ref{lem:robustnp}, and assume $0<\eta<\tfrac18$. Apply
		Lemma~\ref{lem:absnp} with $\mu=\eta$ and let $\xi$ be the resulting constant: for
		large order and $\sigma(H)\ge0.66n$ there is $U\subseteq V(H)$ with $|U|\le\eta n$
		such that $H[U]$ has a perfect matching and $H[U\cup S]$ has one for every
		$S\subseteq V(H)\setminus U$ with $|S|\le\xi n$ and $3\mid|S|$. Take $n_0$ large
		enough for Lemmas~\ref{lem:robustnp} and~\ref{lem:absnp}, for
		$\tfrac{2n}3-2\ge0.66n$, and so that $(1-\eta)n_0$ exceeds the threshold of
		Lemma~\ref{lem:almost} with parameters $\eta$ and $\zeta=\xi/2$.
		
		Put $F=H-U$ and $N=n-|U|$; then $3\mid|U|$, so $3\mid N$, and $N\ge(1-\eta)n$. By
		Lemma~\ref{lem:robustnp} with $R=U$, the hypergraph $F-E(J)$ has a perfect fractional
		matching for every $J\subseteq F$ with $\Delta_1(J)\le\eta n^2$, and $N\le n$ gives
		$\eta N^2\le\eta n^2$, so Lemma~\ref{lem:almost} gives a matching $M$ in $F$ whose
		uncovered set $S$ satisfies
		$|S|\le\xi N/2\le\xi n$; and $|S|=N-3|M|$ is divisible by $3$. The defining property
		of $U$ gives a perfect matching of $H[U\cup S]$, which together with $M$ covers
		$V(H)$.
	\end{proof}

	\medskip\noindent\emph{The partite case.}
	Call a $q$-balanced $3$-partite $H$ \emph{$\varepsilon$-extremal} if there are
	distinct $i,j$ and sets $A_i\subseteq V_i$, $A_j\subseteq V_j$ with
	$\bigl||A_i|-q/2\bigr|\le\varepsilon q$ and $\bigl||A_j|-q/2\bigr|\le\varepsilon q$
	such that at most $\varepsilon q^3$ edges of $H$ avoid $A_i\cup A_j$, and
	\emph{$\varepsilon$-nonextremal} otherwise. As in the non-partite setting the
	definition counts edges outside a barrier and says nothing about edges inside it.

	\begin{lem}\label{lem:robust}
		For every $0<\varepsilon<\tfrac1{10}$ there are $\delta>0$, $\eta>0$ and $q_0$ such
		that the following holds. Let $H$ be an $\varepsilon$-nonextremal $q$-balanced
		$3$-partite $3$-graph with $q\ge q_0$ and
		$\sigma(H)\ge\bigl(\tfrac1{\sqrt2}-\delta\bigr)q$. If $U\subseteq V(H)$ contains the
		same number $s\le\eta q$ of vertices from each class and $J$ is a subhypergraph of
		$H-U$ with $\Delta_1(J)\le\eta q^2$, then $(H-U)-E(J)$ has a perfect fractional
		matching.
	\end{lem}
	
	\begin{proof}
		Put $\alpha=\varepsilon/4$, and let $\beta$ and $Q$ be the
		spectral tolerance and order bound supplied by
		Theorem~\ref{thm:target} for this $\alpha$.
		Decreasing $\beta$ if necessary, assume that
		$0<\beta\le10^{-2}$. Choose
		\[
		\delta=\tfrac\beta2,\qquad
		0<\eta\le\tfrac\varepsilon8\ \text{ with }\ 12\eta+\sqrt\eta\le\tfrac\beta2,
		\]
		and $q_0$ so large that $(1-\eta)q_0\ge Q$.
		
		Let $F=(H-U)-E(J)$, whose classes have the common size $m=q-s\ge(1-\eta)q$. For
		$v\in V(F)$ the link $L_H(v)$ has at most $N=2q$ vertices and spectral radius at
		least $(\tfrac1{\sqrt2}-\delta)q>0.7q$, so deleting the $2s$ vertices of $U$ lying in
		it costs at most $2\cdot2s\cdot2q/(0.7q)<12s\le12\eta q$ by
		Lemma~\ref{lem:basic}(iv), and deleting the at most $\eta q^2$ link edges
		contributed by $J$ costs at most $\sqrt\eta\,q$ by Lemma~\ref{lem:basic}(iii), the
		link being bipartite. Hence
		\[
		\sigma(F)\ \ge\ \Bigl(\tfrac1{\sqrt2}-\delta-12\eta-\sqrt\eta\Bigr)q
		\ \ge\ \Bigl(\tfrac1{\sqrt2}-\beta\Bigr)q\ \ge\ \Bigl(\tfrac1{\sqrt2}-\beta\Bigr)m .
		\]
		
		Suppose $F$ has no perfect fractional matching. Theorem~\ref{thm:target} applied to
		$F$ gives distinct classes and sets $A_i,A_j$ in them with
		$\bigl||A_i|-\tfrac m2\bigr|,\bigl||A_j|-\tfrac m2\bigr|\le\alpha m$, together with
		$D$ of size at most $\alpha m$, such that every edge of $F-D$ meets $A_i\cup A_j$.
		Regarding $A_i$ and $A_j$ as subsets of $V_i$ and $V_j$, their sizes differ from
		$q/2$ by at most $\alpha m+s\le(\alpha+\eta)q<\varepsilon q$. An edge of $H$ avoiding
		$A_i\cup A_j$ either meets $U\cup D$ or lies in $E(J)$; a vertex lies in at most
		$q^2$ legal triples, and $3e(J)=\sum_vd_J(v)\le3q\cdot\eta q^2$, so the number of
		such edges is at most
		\[
		\bigl(3s+|D|\bigr)q^2+e(J)\ \le\ \bigl(3\eta+\alpha+\eta\bigr)q^3
		\ =\ \bigl(4\eta+\alpha\bigr)q^3\ <\ \varepsilon q^3 ,
		\]
		using $\eta\le\varepsilon/8$ and $\alpha=\varepsilon/4$. So $H$ is
		$\varepsilon$-extremal, a contradiction.
	\end{proof}
	
	The passage from robustness to an almost perfect matching is
	Lemma~\ref{lem:almost}, applied class by class: a matching of a $3$-partite
	$3$-graph leaves the same number of vertices uncovered in each class, so a bound on
	the total number uncovered is a bound in each class.
	
	\begin{thm}\label{thm:nonextpart}
		For every $0<\varepsilon<\tfrac1{10}$ there are $\delta>0$ and $q_0$ such that
		every $\varepsilon$-nonextremal $q$-balanced $3$-partite $3$-graph $H$ with
		$q\ge q_0$ and $\sigma(H)\ge(\tfrac1{\sqrt2}-\delta)q$ has a perfect matching.
	\end{thm}
	
	\begin{proof}
		Let $\eta$ be given by Lemma~\ref{lem:robust} for this $\varepsilon$, and apply
		Lemma~\ref{lem:abs} with $\alpha=\eta/2$ to obtain $\mu>0$ and an absorbing matching
		$M_{\mathrm{abs}}$ with at most $\eta q/2$ edges, shrinking $\delta$ so that
		$\delta\le\delta_*$. Put $F=H-V(M_{\mathrm{abs}})$, whose classes have the common
		size $m=q-|M_{\mathrm{abs}}|\ge(1-\eta)q$. Deleting the same number of vertices from
		each class, Lemma~\ref{lem:robust} shows that $F-E(J)$ has a perfect fractional
		matching for every $J\subseteq F$ with $\Delta_1(J)\le\eta q^2$, so
		Lemma~\ref{lem:almost}, applied with $\eta/9$ in place of $\eta$ and with
		$\zeta=\mu/2$, gives a matching in $F$ whose uncovered set $W$ has
		$|W\cap V_i|\le\mu q$ for each $i$; the hypothesis is met because $F$ has $N=3m$
		vertices with $m\le q$, so $\tfrac\eta9N^2=\eta m^2\le\eta q^2$; the three numbers are equal because
		$F$ is balanced and every edge meets each class once. The defining property of
		$M_{\mathrm{abs}}$ turns $V(M_{\mathrm{abs}})\cup W$ into a perfect matching, which
		together with the matching in $F$ covers $V(H)$.
	\end{proof}
	
	\subsection{The extremal case}\label{sec:extremal}
	
	\medskip\noindent\emph{The non-partite case.}
	Here the order of the hypergraph is $3m$. If $A,B$ partition the vertex set, an
	\emph{$AAB$ triple} is one with exactly two vertices in $A$ and one in $B$, and
	$AAA$, $ABB$, $BBB$ are read the same way; when the partition changes we say which
	one a type refers to. We next prove three lemmas for the non-partite extremal case.
	
	\begin{lem}\label{lem:matchdel}
		Let $H$ be a $3$-graph on $3m$ vertices with $\sigma(H)>2m-2$. If $k\ge1$ and
		$m\ge30k$, then $H-C$ has a matching of size $k$ for every $C\subseteq V(H)$ of
		order $m-k$.
	\end{lem}
	
	\begin{proof}
		We first record a degree estimate. Let $G$ be a graph on $3m-1$ vertices with
		$\rho(G)>L:=2m-2$, let $R\subseteq V(G)$ have order $r$ with $1\le r\le m-1$, and
		put $t=m-1-r$; we claim
		\begin{equation}\label{eq:npdeg}
			e(G-R)>\tfrac12t(3m+t-1)\ \ge\ mt .
		\end{equation}
		Indeed, put $J=G-R$ and $Q=K_R\vee J$, so $\lambda:=\rho(Q)\ge\rho(G)>L$. The Perron
		vector of $Q$ is constant, say $a$, on $R$; writing $y_v$ for its other coordinates
		and $s$ for the sum of all coordinates, the eigenvalue equations give
		$(\lambda+1)a=s$ and $(\lambda+1)y_v\le s$, so $y_v\le a$, and
		\[
		(\lambda-r+1)a=\sum_{v\in V(J)}y_v,\qquad
		\lambda\sum_{v\in V(J)}y_v=r|V(J)|a+\sum_vd_J(v)y_v\le\bigl(r|V(J)|+2e(J)\bigr)a .
		\]
		Hence $\lambda(\lambda-r+1)\le r(3m-1-r)+2e(J)$, and since $r\le m-1$ the function
		$x(x-r+1)$ is increasing for $x\ge L$, which gives \eqref{eq:npdeg}.
		
		Now put $F=H-C$ and $N=|V(F)|=2m+k$, and suppose $\nu(F)\le k-1$. Let
		$S=\{v\in V(F):d_F(v)>3kN\}$ and $s=|S|$. If $s\ge k$, choose $k$ vertices of $S$ and
		cover them successively by disjoint edges: when the $j$th is treated, at most
		$3(j-1)+(k-j)\le3k-3$ vertices are forbidden and each lies in fewer than $N$ edges at
		the current vertex, so fewer than $3kN$ edges are excluded while the current vertex
		has degree more than $3kN$. This gives $k$ disjoint edges, so $s\le k-1$; the same
		comparison, started from a maximum matching of $F-S$, gives $\nu(F-S)\le k-s-1=:t$.
		
		For $v\in V(F)\setminus S$, apply \eqref{eq:npdeg} to $L_H(v)$ with the deleted set
		$C\cup S$, of order $m-k+s=m-1-t$; this gives $d_{F-S}(v)>mt$. If $t=0$ this
		contradicts $\nu(F-S)=0$. If $t>0$, the vertices of a maximum matching of $F-S$ form
		a vertex cover of at most $3t$ vertices, each of degree at most $3kN$, so
		$e(F-S)\le9ktN$, while summing $d_{F-S}(v)>mt$ over the $N-s$ vertices outside $S$
		and dividing by three gives $e(F-S)>(N-s)mt/3$. Cancelling $t$ leaves
		$m<27kN/(N-s)<30k$, a contradiction.
	\end{proof}
	
	\begin{lem}\label{lem:ext1}
		There exist $\varepsilon_1>0$ and $m_0$ such that the following holds. Let $H$ be a
		$3$-graph on $3m$ vertices with $m\ge m_0$ and $\sigma(H)>2m-2$, and let
		$A_0\subseteq V(H)$ satisfy $|A_0|=m-1$ and
		$e\bigl(H[V(H)\setminus A_0]\bigr)\le\varepsilon_1m^3$. Then $H$ has a perfect
		matching.
	\end{lem}
	
	\begin{proof}
		Write $B_0=V(H)\setminus A_0$ and
		$T^*(A_0)=\{e\in\binom{V(H)}3:1\le|e\cap A_0|\le2\}$.
		
		\medskip\noindent\emph{Claim 1.} Let $m\ge4$, let $\beta>0$, and let $H$ be a $3$-graph on $3m$ vertices with
		$\sigma(H)>2m-2$ and $V(H)=A_0\sqcup B_0$, where $|A_0|=m-1$. Put
		$S_x=K_{A_0\setminus\{x\}}\vee I_{B_0\setminus\{x\}}$. If a vertex $x$ satisfies
		$e\bigl(L_H(x)[B_0\setminus\{x\}]\bigr)\le\beta m^2$, then
		$|E(S_x)\setminus E(L_H(x))|<300\sqrt\beta\,m^2$.
		
		Delete from $L_H(x)$ all edges inside $B_0\setminus\{x\}$ and call the result $K$, so
		that $K\subseteq S_x$; put $\lambda_0=\rho(S_x)$ and $D=\lambda_0-\rho(K)\ge0$. The
		complete part of $S_x$ has order $p\in\{m-2,m-1\}$ and its independent part has order
		$q=3m-1-p$, and the positive eigenvalue satisfies
		\begin{equation}\label{eq:sxeig}
			\lambda_0^2-(p-1)\lambda_0-pq=0 .
		\end{equation}
		For $p=m-1$ the positive root of \eqref{eq:sxeig} is $2m-2$; for $p=m-2$ the left
		side at $\lambda_0=m$ equals $-2m^2+6m+2<0$ for $m\ge4$, so the positive root is at
		least $m$. As the graph with $p=m-2$ is a spanning subgraph of the one with
		$p=m-1$ after relabeling, $m\le\lambda_0\le2m-2$, and Lemma~\ref{lem:basic}(iii) gives
		\[
		D<\lambda_0-(2m-2)+\sqrt{2\beta}\,m\le\sqrt{2\beta}\,m .
		\]
		Let $\mathbf u$ be the positive unit Perron vector of $S_x$, constant on each part.
		If its value on the independent part is $b$, the eigenvalue equation at an
		independent vertex gives $\lambda_0b/p$ on the complete part, which is at least $b$.
		Normalizing,
		\[
		\min_iu_i^2=\frac1{\lambda_0^2/p+q}\ \ge\ \frac1{8m},
		\qquad\text{since}\quad
		\frac{\lambda_0^2}p+q\le\frac{(2m-2)^2}{m-2}+2m+1=6m+1+\frac4{m-2}\le8m .
		\]
		Every eigenvalue of $S_x$ other than $\lambda_0$ is nonpositive: vectors of
		coordinate sum zero within the complete part give $-1$, those within the independent
		part give $0$, and the remaining two eigenvalues are the roots of \eqref{eq:sxeig},
		whose product $-pq$ is negative. Lemma~\ref{lem:missing} with $N=3m-1$ therefore
		gives
		\begin{align*}
		|E(S_x)\setminus E(K)|&\le8mD\Bigl(\frac{8(3m-1)}{\lambda_0}+2\Bigr)
		\le8mD(24+2)\\
		&=208mD<208\sqrt2\,\sqrt\beta\,m^2<300\sqrt\beta\,m^2 ,
		\end{align*}
		and $E(S_x)\setminus E(K)=E(S_x)\setminus E(L_H(x))$ because the deleted edges lie
		outside $S_x$.
		
		\medskip\noindent\emph{Step 1: few triples of $T^*(A_0)$ are missing.} Let $B_0=V(H)\setminus A_0$ and let $\zeta$ be the constant fixed below. Put
		$\nu=(\zeta/1800)^2$ and $\varepsilon_1=\zeta\nu/18$, and take $m_0$ large enough for Claim 1 with parameter $\nu$. For $b\in B_0$
		write $s_b=e(L_H(b)[B_0\setminus\{b\}])$; each edge of $H[B_0]$ contributes once to
		$s_b$ for each of its three vertices, so
		$\sum_{b\in B_0}s_b=3e(H[B_0])\le3\varepsilon_1m^3$ and at most
		$(3\varepsilon_1/\nu)m$ vertices have $s_b>\nu m^2$. For the others,
		Claim 1 shows that fewer than $300\sqrt\nu\,m^2$ edges of
		$K_{A_0}\vee I_{B_0\setminus\{b\}}$ are missing from $L_H(b)$, while for an
		exceptional vertex we use
		$e(K_{A_0}\vee I_{B_0\setminus\{b\}})=\binom{m-1}2+2m(m-1)<3m^2$. Summing over
		$b\in B_0$, a missing triple of type $A_0A_0B_0$ is counted once and one of type
		$A_0B_0B_0$ twice, and these are exactly the types in $T^*(A_0)$, so
		\[
		|T^*(A_0)\setminus E(H)|\ \le\ (2m+1)300\sqrt\nu\,m^2+\frac{3\varepsilon_1}{\nu}m\cdot3m^2
		\ \le\ \Bigl(900\sqrt\nu+\frac{9\varepsilon_1}{\nu}\Bigr)m^3\ =\ \zeta m^3 ,
		\]
		using $2m+1\le3m$. It remains to treat this situation.
		
		\medskip\noindent\emph{Step 2: repairing the partition.} Choose the constants in the order
		\begin{equation}\label{eq:extconst}
			\eta=\frac1{1024},\qquad 0<\beta\le\Bigl(\frac\eta{300}\Bigr)^2,\qquad
			0<\kappa\le\min\Bigl\{\frac\beta{100},\frac1{1000}\Bigr\},\qquad
			\zeta=\frac{\eta\kappa}3 ,
		\end{equation}
		and let $m_0$ be large enough that Claim 1 applies with parameter $\beta$ and $m_0\ge1/\kappa$.
		
		For $x\in V(H)$ let $d_{\mathrm{miss}}(x)$ be the number of triples of
		$T^*(A_0)\setminus E(H)$ containing $x$, and put
		$W=\{x:d_{\mathrm{miss}}(x)>\eta m^2\}$ and $w=|W|$. Each missing triple has three
		vertices, so $w\eta m^2\le\sum_xd_{\mathrm{miss}}(x)=3|T^*(A_0)\setminus E(H)|\le3\zeta m^3$
		and hence $w\le\kappa m$. Among the vertices of $W$ distinguish the \emph{strong}
		ones,
		\[
		P=\bigl\{x\in W:\ e\bigl(L_H(x)[B_0\setminus\{x\}]\bigr)\ge\beta m^2\bigr\},
		\]
		the others being \emph{weak}.
		
		Adjust the partition so that every strong vertex lies in the smaller part and every
		weak vertex in the larger. Put $A_1=(A_0\setminus W)\cup P$; if $|A_1|\le m$ take
		$A=A_1$, and otherwise delete $|A_1|-m$ vertices of $A_0\setminus W$ from $A_1$,
		which is possible because $|P|\le w<m$. Let $B=V(H)\setminus A$ and $|A|=m-k$. Then
		\begin{equation}\label{eq:adjust}
			P\subseteq A,\qquad 0\le k\le w+1\le2\kappa m,\qquad|A\,\triangle\,A_0|\le2w .
		\end{equation}
		For the last bound write $p_B=|P\cap B_0|$ and $a_W=|(W\setminus P)\cap A_0|$, so
		that $|A_1|=m-1-a_W+p_B$; without the deletion the symmetric difference has size
		$p_B+a_W\le w$, and with it the size is $p_B+a_W+(p_B-a_W-1)=2p_B-1\le2w$.
		
		We claim that for every $x\notin P$ the number of missing $ABB$ triples containing
		$x$, with respect to the new partition, is at most
		\begin{equation}\label{eq:missABB}
			(\eta+6\kappa)m^2 .
		\end{equation}
		If $x\in A\setminus P$ then $x\in A_0\setminus W$, and a candidate pair avoiding
		$A\,\triangle\,A_0$ lies in $B_0$, so the triple had old type $A_0B_0B_0$ and at
		most $d_{\mathrm{miss}}(x)\le\eta m^2$ of these are missing. If $x\in B\setminus W$,
		a candidate pair avoiding the symmetric difference has one vertex in $A_0$ and one
		in $B_0$, and the triple has old type $A_0B_0B_0$ or $A_0A_0B_0$; both lie in
		$T^*(A_0)$, so the same bound holds. If $x\in W\setminus P$ then $x\in B$ and
		$e(L_H(x)[B_0\setminus\{x\}])<\beta m^2$, so by Claim 1 fewer than
		$300\sqrt\beta\,m^2\le\eta m^2$ edges of $K_{A_0\setminus\{x\}}\vee I_{B_0\setminus\{x\}}$
		are absent from $L_H(x)$, and a candidate pair avoiding the symmetric difference is
		an edge of that split graph. Pairs meeting $A\,\triangle\,A_0$ number at most
		$3m|A\,\triangle\,A_0|\le6\kappa m^2$, which proves \eqref{eq:missABB}.
		
		For $p\in P$, at most $w$ vertices of $B_0$ have moved to $A$, and deleting the pairs
		meeting them removes at most $3wm$ link edges, so
		\begin{equation}\label{eq:strongdeg}
			e\bigl(L_H(p)[B]\bigr)\ \ge\ (\beta-3\kappa)m^2 .
		\end{equation}
		If $k>0$ then $30k\le60\kappa m\le m$, so Lemma~\ref{lem:matchdel} applied with the
		deleted set $A$ gives a matching $M_0$ of size $k$ in $H[B]$; if $k=0$ put
		$M_0=\emptyset$. Starting from $M_0$, cover the vertices of $P$ one at a time by
		disjoint $ABB$ edges, the $A$-vertex of the edge chosen for $p$ being $p$ itself. At
		each step at most $3k+2|P|\le8\kappa m$ vertices of $B$ have been used, so at most
		$3m\cdot8\kappa m=24\kappa m^2$ pairs of $L_H(p)[B]$ meet them, and by
		\eqref{eq:strongdeg} at least $(\beta-27\kappa)m^2>0$ choices remain.
		
		After deleting $M_0$ and the $|P|$ chosen edges, the remaining parts have sizes $t$
		and $2t$ with $t=m-k-|P|\ge(1-3\kappa)m\ge m/2\ge16$: indeed $M_0$ uses $3k$ vertices
		of $B$ and the other edges use $|P|$ of $A$ and $2|P|$ of $B$, and $|B|=2m+k$, so
		$B$ retains $2m+k-3k-2|P|=2t$ vertices. Split the remaining $B$ into two sets of
		size $t$. Every remaining vertex lies outside $P$, so \eqref{eq:missABB} bounds its
		number of missing triples transversal to the three parts, and by
		\eqref{eq:extconst}, $(\eta+6\kappa)m^2\le m^2/64\le t^2/16$. Lemma~\ref{lem:box}
		now gives a perfect matching on the remaining vertices.
	\end{proof}
	
	\begin{lem}\label{lem:ext2}
		There exist $\varepsilon_2>0$ and $m_2$ such that the following holds. Let $H$ be a
		$3$-graph on $3m$ vertices with $m\ge m_2$ and $\sigma(H)>2m-2$, and let
		$V(H)=A_0\sqcup B_0$ with $|A_0|=2m-1$. If
		$|\{e\in E(H):|e\cap A_0|\le1\}|\le\varepsilon_2m^3$, then $H$ has a perfect
		matching.
	\end{lem}
	
	\begin{proof}
		\medskip\noindent\emph{Claim.} Let $H$ be a $3$-graph on $3m$ vertices with $V(H)=A\sqcup B$, $|A|=2m-k$,
		$|B|=m+k$ and $1\le k\le m/600$. If $d_H(b)>\binom{2m-1}2$ for every $b\in B$, then
		there is a matching $M$ of $ABB$ and $BBB$ edges with $|M|\le k$ and
		$w(M)\in\{k,k+1\}$, where an $ABB$ edge has weight $1$, a $BBB$ edge has weight $2$,
		and $w(M)=\sum_{e\in M}w(e)$.
		
		Put $p=2m-k$ and $q=m+k$, and let $F$ be the spanning subhypergraph of $H$ whose
		edges are those of types $ABB$ and $BBB$. If $F$ has a matching of weight at least
		$k$, order its edges and retain the initial segment whose weight first reaches $k$;
		since every weight is $1$ or $2$ that segment has weight $k$ or $k+1$, and before
		its last edge its weight is at most $k-1$, so it has at most $k$ edges. It therefore
		suffices to produce a matching of weight at least $k$ in $F$. Suppose every matching
		in $F$ has weight at most $k-1$.
		
		Put $L=30km$ and
		\begin{gather*}
		S_A=\{a\in A:d_{ABB}(a)>L\},\qquad S_2=\{b\in B:d_{BBB}(b)>L\},\\
		S_1=\{b\in B\setminus S_2:d_{ABB}(b)>L\} ,
		\end{gather*}
		give capacity $1$ to the vertices of $S_A\cup S_1$ and $2$ to those of $S_2$, write
		$c(x)$ for the capacity, and set $a=|S_A|$, $s=|S_1|$, $c=|S_2|$,
		$S=S_A\cup S_1\cup S_2$ and $C=\sum_{x\in S}c(x)=a+s+2c$.
		
		We claim $C\le k-1$. Otherwise choose vertices of $S$ one at a time until their total
		capacity first reaches $k$; the resulting set has some size $\ell\le k$ and capacity
		$k$ or $k+1$. Process its vertices in any order, choosing through a vertex of
		$S_A\cup S_1$ an $ABB$ edge and through a vertex of $S_2$ a $BBB$ edge, avoiding all
		edges already chosen and all selected vertices not yet processed. When the $j$th
		vertex is processed at most $3(j-1)+(\ell-j)<4k$ vertices are forbidden, and a
		forbidden vertex lies with the current one in fewer than $3m$ edges, so fewer than
		$12km<L$ candidates are excluded while the current vertex has more than $L$ edges of
		the required type. The chosen edges form a matching in which each edge has weight
		equal to the capacity of its vertex, so its weight is at least $k$, a contradiction.
		
		For $e\in E(F)$ put $w'(e)=w(e)-\sum_{x\in e\cap S}c(x)$, let $R$ be the spanning
		subhypergraph of the edges with $w'(e)>0$, and put $t=k-1-C\ge0$; each edge of $R$
		has residual weight $1$ or $2$. Every matching in $R$ has residual weight at most
		$t$. Otherwise retain the initial segment $Q$ whose residual weight first reaches
		$t+1$, so that $w'(Q)\in\{t+1,t+2\}$ and $|Q|\le t+1=k-C$. For each vertex of
		$S\setminus V(Q)$ add an edge of the type prescribed by its capacity, successively
		avoiding $V(Q)$, the edges already added and the vertices of $S\setminus V(Q)$ not
		yet processed. With $r=|Q|$ and $h_0=|S\setminus V(Q)|\le C$ we have $r+h_0\le k$, so
		at the $j$th added edge at most $3r+3(j-1)+(h_0-j)\le3(r+h_0)-3<4k$ vertices are
		forbidden and the same estimate $12km<L$ applies. Each added edge contains its
		prescribed vertex of $S$ and no other, so the added edges contribute exactly
		$\sum_{x\in S\setminus V(Q)}c(x)$; since $w(Q)=w'(Q)+\sum_{x\in V(Q)\cap S}c(x)$, the
		enlarged matching has weight $w'(Q)+C\ge k$, a contradiction. As every residual
		weight is at least one, this also gives
		\begin{equation}\label{eq:nuR}
			\nu(R)\le t .
		\end{equation}
		
		Every edge of $F$ at a vertex of $S_A$ is an $ABB$ edge of weight $1$, and the
		capacity of a vertex of $S_2$ is $2$, at least the weight of any edge of $F$; so the
		vertices of $S_A\cup S_2$ are isolated in $R$. An edge of $R$ at a vertex of $S_1$ is
		a $BBB$ edge, of which there are at most $L$ there; a vertex of $A\setminus S_A$ lies
		in at most $L$ edges of $F$; and a vertex of $B\setminus(S_1\cup S_2)$ has both
		degrees at most $L$. Hence $\Delta_1(R)\le2L$, and since the vertices of a maximum
		matching of $R$ cover its edges,
		\begin{equation}\label{eq:eR}
			e(R)\le3\nu(R)\Delta_1(R)\le6Lt=180kmt .
		\end{equation}
		
		Let $b\in B\setminus(S_1\cup S_2)$. An edge of $H$ at $b$ outside $F$ has its other
		two vertices in $A$, so using $2m-1=p+k-1$ and
		$\binom{p+u}2-\binom p2=pu+\binom u2$ with $u=k-1$,
		\begin{equation}\label{eq:Pb}
			d_F(b)>\binom{2m-1}2-\binom p2=p(k-1)+\binom{k-1}2=:P .
		\end{equation}
		Write $h=s+c$. An $ABB$ edge at $b$ has nonpositive residual weight exactly when its
		$A$ vertex lies in $S_A$ or its other $B$ vertex lies in $S_1\cup S_2$, and splitting
		by whether the $A$ vertex lies in $S_A$ bounds these by $a(q-1)+(p-a)h$. A $BBB$ edge
		at $b$ has nonpositive residual weight exactly when its other two vertices include
		one of $S_2$, of which there are at most $c(q-1)$, or both lie in $S_1$, of which
		there are at most $\binom s2$. So at most
		$U=(a+c)(q-1)+(p-a)h+\binom s2$ edges at $b$ lie in $F\setminus R$, and with
		$t=k-1-a-s-2c$,
		\begin{equation}\label{eq:PU}
			P-U=pt+(a+c)(p-q+1)+ah+\binom{k-1}2-\binom s2\ \ge\ pt ,
		\end{equation}
		since $p-q+1=m-2k+1>0$, the quantities $a,c,h$ are nonnegative and
		$0\le s\le C\le k-1$. By \eqref{eq:Pb} and \eqref{eq:PU},
		\begin{equation}\label{eq:dRb}
			d_R(b)>pt\qquad\text{for every }b\in B\setminus(S_1\cup S_2),
		\end{equation}
		and there are at least $q-h\ge m+k-(k-1)=m+1$ such vertices.
		
		If $t=0$ then \eqref{eq:eR} says $R$ has no edges while \eqref{eq:dRb} produces a
		vertex of positive degree. If $t>0$ then, using $m$ of these vertices and
		$p=2m-k\ge m$,
		\[
		m^2t<\sum_{b\in B\setminus(S_1\cup S_2)}d_R(b)\le3e(R)\le540kmt ,
		\]
		so $m<540k$, contradicting $m\ge600k$.
		
		\medskip
		
		Put $\eta=1/1024$, $\kappa=\eta/1000$ and choose $\varepsilon_2>0$ with
		\begin{equation}\label{eq:eps2}
			\frac{12\sqrt3\sqrt{\varepsilon_2}}\eta+20\varepsilon_2\ \le\ \kappa ,
		\end{equation}
		for instance $\varepsilon_2=\min\{(\kappa\eta/(24\sqrt3))^2,\kappa/40\}$, and take
		$m_2=\lceil\max\{100,1/\kappa\}\rceil$. Write $L=2m-2$. By Lemma~\ref{lem:basic}(ii),
		$2d_H(x)\ge\rho(L_H(x))^2+\rho(L_H(x))>L(L+1)$, so
		\begin{equation}\label{eq:degx}
			d_H(x)>\binom{2m-1}2\qquad\text{for every }x\in V(H).
		\end{equation}
		
		For $b\in B_0$ let $c_b$ be the number of edges of $L_H(b)$ not contained in $A_0$,
		and put $J_b=L_H(b)[A_0]$ and $D_b=\sqrt{2c_b}$. Viewing $J_b$ as a graph on
		$V(H)\setminus\{b\}$, Lemma~\ref{lem:basic}(iii) gives $\rho(J_b)>L-D_b$. If $D_b<L$
		then Lemma~\ref{lem:basic}(ii) and the strict increase of $s^2+s$ for $s\ge0$ give
		\[
		\binom{2m-1}2-e(J_b)<\frac{L(L+1)-(L-D_b)(L-D_b+1)}2=\frac{(2L+1)D_b-D_b^2}2\le2mD_b ,
		\]
		and if $D_b\ge L$ the same bound follows from $\binom{2m-1}2\le2mD_b$. Let
		$\mathcal D$ be the family of missing triples with exactly two vertices in $A_0$;
		each has a unique vertex in $B_0$, so summing the last display over $b\in B_0$ gives
		$|\mathcal D|$ on the left. An edge counted in $c_b$ has at most one vertex in $A_0$
		and is counted at most three times, so $\sum_bc_b\le3\varepsilon_2m^3$, and by
		Cauchy--Schwarz
		\begin{equation}\label{eq:Dsize}
			|\mathcal D|\le2\sqrt2\,m\sum_{b\in B_0}\sqrt{c_b}
			\le2\sqrt2\,m\sqrt{(m+1)\textstyle\sum_bc_b}\le4\sqrt{3\varepsilon_2}\,m^3 .
		\end{equation}
		For $a\in A_0$ put $p_a=e(L_H(a)[A_0\setminus\{a\}])$ and $q_a=e(L_H(a)[B_0])$.
		Splitting the adjacency matrix of $L_H(a)$ into its bipartite part between
		$A_0\setminus\{a\}$ and $B_0$ and the rest, and bounding the operator norm of the
		second by its Frobenius norm,
		\[
		\rho(L_H(a))\ \le\ \sqrt{(2m-2)(m+1)}+\sqrt{2(p_a+q_a)}\ <\ \sqrt2\,m+\sqrt{2(p_a+q_a)} .
		\]
		If $p_a<m^2/10$ and $q_a\le m^2/20$ the right side is less than
		$(\sqrt2+\sqrt{3/10})m<1.97m\le2m-2$ for $m\ge100$, contradicting the hypothesis. So
		every $a$ with $p_a<m^2/10$ has $q_a>m^2/20$, and since $\sum_aq_a$ counts each
		$A_0B_0B_0$ edge once,
		\begin{equation}\label{eq:pa}
			|\{a\in A_0:p_a<m^2/10\}|\ \le\ 20\varepsilon_2m .
		\end{equation}
		Writing $d_{\mathcal D}(x)$ for the number of members of $\mathcal D$ containing $x$
		and setting
		$W=\{x:d_{\mathcal D}(x)>\eta m^2\}\cup\{a\in A_0:p_a<m^2/10\}$, the identity
		$\sum_xd_{\mathcal D}(x)=3|\mathcal D|$ with \eqref{eq:Dsize}, \eqref{eq:pa} and
		\eqref{eq:eps2} gives $w:=|W|\le\kappa m$.
		
		For $x\in W$ let $c_0(x)$ be the number of link edges of $L_H(x)$ with at least one
		end in $B_0$, put $S=\{x\in W:c_0(x)\ge\eta m^2\}$, and set
		$A=(A_0\setminus W)\cup S$, $B=V(H)\setminus A$, $|A|=2m-k$, $|B|=m+k$. All vertices
		that change parts lie in $W$, so $|A\,\triangle\,A_0|\le w$ and
		$k=1+|A_0\cap W|-|S|$, whence $|k|\le w+1\le2\kappa m$; note $k$ may be negative.
		
		Every $x\notin S$ lies in at most
		\begin{equation}\label{eq:missAAB}
			(\eta+3\kappa)m^2
		\end{equation}
		missing $AAB$ triples for the new partition. If $x\notin W$ it has not changed parts
		and $d_{\mathcal D}(x)\le\eta m^2$; a candidate pair avoiding $A\,\triangle\,A_0$
		gives an old $A_0A_0B_0$ triple, and pairs meeting the symmetric difference number at
		most $3mw\le3\kappa m^2$. If $x\in W\setminus S$ then $x\in B$ and $c_0(x)<\eta m^2$,
		so by \eqref{eq:degx}, $e(L_H(x)[A_0\setminus\{x\}])>\binom{2m-1}2-\eta m^2$; the
		number of available old $A_0A_0$ pairs is $\binom{2m-1}2$ if $x\in B_0$ and
		$\binom{2m-2}2$ if $x\in A_0$, so in either case fewer than $\eta m^2$ old $A_0A_0$
		pairs are missing from the link, and again at most $3mw$ changed pairs are added.
		
		For $x\in S\subseteq A$, at least $\eta m^2$ old link pairs meet $B_0$, and a pair
		ceases to meet the new $B$ only if it meets a vertex that changed parts, so at most
		$3mw$ are lost and $x$ lies in at least
		\begin{equation}\label{eq:strong2}
			(\eta-3\kappa)m^2
		\end{equation}
		edges meeting the new $B$; as $x\in A$ these have type $AAB$ or $ABB$.
		
		Give an edge the weight $\omega(e)=2-|e\cap A|=|e\cap B|-1$, so that $AAA$, $AAB$,
		$ABB$ and $BBB$ edges have weights $-1,0,1,2$, and put $\omega(M)=\sum_{e\in M}\omega(e)$.
		If $k>0$ then $1\le k\le2\kappa m<m/600$, and \eqref{eq:degx} allows
		the Claim to be applied to the new partition, giving a matching $M_0$ of
		$ABB$ and $BBB$ edges with $|M_0|\le k$ and $\omega(M_0)\in\{k,k+1\}$; if $k\le0$ put
		$M_0=\emptyset$. In both cases $|M_0|\le2\kappa m$ and $\omega(M_0)\ge k$. Starting
		from $M_0$, repeatedly take an uncovered $x\in S$ and cover it by an edge counted in
		\eqref{eq:strong2} avoiding the current matching; the matching always has at most
		$2\kappa m+|S|\le3\kappa m$ edges and covers at most $9\kappa m$ vertices, so at most
		$27\kappa m^2$ candidate pairs are blocked and at least $(\eta-30\kappa)m^2>0$ remain.
		The added edges have type $AAB$ or $ABB$, so the resulting matching $M_1$ satisfies
		$|M_1|\le3\kappa m$ and $\omega(M_1)\ge k$.
		
		Let $d=\omega(M_1)-k\ge0$; since every edge has weight at most $2$ and
		$|k|\le2\kappa m$, we have $d\le2|M_1|+|k|\le8\kappa m$. Add $d$ disjoint $AAA$ edges,
		each decreasing the weight by one. At any stage the current matching $M$ has
		$|M|\le11\kappa m$, so $A_0\setminus(W\cup V(M))$ has at least
		$(2m-1)-w-3|M|\ge(2-34\kappa)m-1\ge m$ vertices; choosing $a$ there, the definition of
		$W$ gives $p_a\ge m^2/10$, so $a$ lies in at least $m^2/10$ old $A_0A_0A_0$ edges,
		of which at most $2m(w+3|M|)\le68\kappa m^2<m^2/10$ meet $W\cup V(M)$ in another
		vertex. The final matching $M$ covers $S$ and has $|M|\le11\kappa m$, $\omega(M)=k$.
		
		Put $t=m-|M|$. Since $|e\cap B|=1+\omega(e)$,
		\[
		|B\setminus V(M)|=m+k-\sum_{e\in M}|e\cap B|=m+k-|M|-\omega(M)=m-|M|=t ,
		\]
		and there are $3t$ uncovered vertices in all, so $|A\setminus V(M)|=2t$; also
		$t\ge(1-11\kappa)m\ge m/2\ge16$. Partition $A\setminus V(M)$ into two sets of size
		$t$ and let $U=B\setminus V(M)$. Every remaining vertex lies outside $S$, so by
		\eqref{eq:missAAB} it lies in at most $(\eta+3\kappa)m^2\le m^2/64\le t^2/16$ missing
		transversal triples, and Lemma~\ref{lem:box} finishes the proof.
	\end{proof}

	\medskip\noindent\emph{The partite case.}
	Let $H$ be a $q$-balanced $3$-partite $3$-graph with
	classes $X,Y,Z$. For sets $A\subseteq X$ and
	$B\subseteq Y$, put $X_0=X\setminus A$ and
	$Y_0=Y\setminus B$.
	A legal triple has \emph{type} $ij$ with $i=\mathbf1[x\in A]$ and
	$j=\mathbf1[y\in B]$, it is \emph{crossing} if its type is $10$ or $01$, and
	$E_{ij}$ denotes the set of edges of that type. Write $m_{A,B}(v)$ for the number of
	missing crossing triples at $v$, and $t_{ij}(M)$ for the number of edges of type
	$ij$ in a matching $M$. A perfect matching must satisfy
	$t_{11}-t_{00}=|A|+|B|-q$, since each of its edges uses at most one vertex of $A$
	and at most one of $B$; this bookkeeping is what the argument has to arrange.
	
	\begin{lem}\label{lem:crossmiss}
		Let $0<\varepsilon\le\tfrac1{12}$ and $0<\delta\le\tfrac1{10}$, and suppose
		$\bigl||A|-q/2\bigr|\le\varepsilon q$, $\bigl||B|-q/2\bigr|\le\varepsilon q$,
		$|E_{00}|\le\varepsilon q^3$ and $\sigma(H)\ge(\tfrac1{\sqrt2}-\delta)q$. Then each
		crossing region has at most $(2\varepsilon+2\delta)q^3$ missing edges, and for every
		$\xi>0$ the set $\mathcal B=\{v:m_{A,B}(v)>\xi q^2\}$ has
		$|\mathcal B|\le12(\varepsilon+\delta)q/\xi$.
	\end{lem}
	
	\begin{proof}
		By Lemma~\ref{lem:basic}(ii),
		\begin{equation}\label{eq:degcross}
			d_H(v)=e(L_H(v))\ge\rho(L_H(v))^2\ge\bigl(\tfrac12-\sqrt2\,\delta+\delta^2\bigr)q^2
			\ge\bigl(\tfrac12-2\delta\bigr)q^2 .
		\end{equation}
		Put $a=|A|$, $b=|B|$, and let $m_1,m_2$ count the missing edges in $A\times Y_0\times Z$
		and $X_0\times B\times Z$. Every edge meeting $Y_0$ has type $10$ or $00$ and is
		counted once in the sum of degrees over $Y_0$, so by \eqref{eq:degcross}
		\[
		m_1=a(q-b)q-\sum_{y\in Y_0}d_H(y)+|E_{00}|
		\le(q-b)\bigl(aq-(\tfrac12-2\delta)q^2\bigr)+\varepsilon q^3
		\le(2\varepsilon+2\delta)q^3 ,
		\]
		and summing degrees over $X_0$ gives the same bound for $m_2$. The two regions are
		disjoint and each missing crossing triple is counted at each of its three vertices,
		so $\sum_vm_{A,B}(v)=3(m_1+m_2)\le12(\varepsilon+\delta)q^3$; every vertex of
		$\mathcal B$ contributes more than $\xi q^2$.
	\end{proof}
	
	\begin{lem}\label{lem:complete}
		Suppose $q\ge48$ and $\tfrac{5q}{12}\le|A|,|B|\le\tfrac{7q}{12}$. Let $M$ be a
		matching of $H$ with $|M|\le q/12$ and $t_{11}(M)-t_{00}(M)=|A|+|B|-q$ such that
		every vertex outside $V(M)$ satisfies $m_{A,B}(v)\le q^2/144$. Then $M$ extends to a
		perfect matching of $H$.
	\end{lem}
	
	\begin{proof}
		Delete $V(M)$ and put $q'=q-|M|$, so the residual classes $X',Y',Z'$ have size $q'$.
		With $A'=A\cap X'$ and $B'=B\cap Y'$, an edge of type $10$ or $11$ uses one vertex of
		$A$ and no other edge does, and similarly for $B$, so
		$|A'|=|A|-t_{10}(M)-t_{11}(M)$ and $|B'|=|B|-t_{01}(M)-t_{11}(M)$; since
		$|M|=t_{00}+t_{10}+t_{01}+t_{11}$, the hypothesis gives
		$|A'|+|B'|-q'=|A|+|B|-q+t_{00}(M)-t_{11}(M)=0$. Hence $|X'\setminus A'|=|B'|$ and
		$|Y'\setminus B'|=|A'|$, so the two crossing boxes
		$A'\times(Y'\setminus B')\times Z'$ and $(X'\setminus A')\times B'\times Z'$ can be
		completed to two $3$-partite $3$-graphs with classes of equal sizes $|A'|$ and
		$|B'|$, after splitting $Z'$ accordingly. Each has class size at least
		$\tfrac{5q}{12}-\tfrac q{12}=\tfrac q3\ge16$, and every vertex lies in at most
		$q^2/144\le(q/3)^2/16$ missing transversal triples, so Lemma~\ref{lem:box} supplies a
		perfect matching in each. Their vertex sets are disjoint and together with $V(M)$
		cover $V(H)$.
	\end{proof}
	
	For sets $Y,Z$ of size $q$ and $C\subseteq Y$, $C'\subseteq Z$ with $|C|=c\ge1$ and
	$|C'|=c'\ge0$, let $Q(c,c')$ be the bipartite graph on $Y\cup Z$ whose edges are the
	pairs $yz$ with $y\in C$ or $z\in C'$, and put $g(c,c')=\rho(Q(c,c'))^2$; thus
	$Q(c,c')$ is the graph $J_q(c,c')$ of Lemma~\ref{lem:blocks} and
	$g(c,c')=\rho(J_q(c,c'))^2$. If every edge of $H$ at $v\in V_1$ meets
	$C\subseteq V_2$ or $C'\subseteq V_3$, then $L_H(v)\subseteq Q(c,c')$.
	
	\begin{thm}\label{thm:defect}
		Let $q\ge200$ and let $H$ be a $q$-balanced $3$-partite $3$-graph with
		$\sigma(H)>\tau(q)$. Let $A_1\subseteq V_1$ and $A_2\subseteq V_2$ satisfy
		$|A_1|,|A_2|\ge0.4q$ and $|A_1|+|A_2|=q-1-s$, where $0\le s\le\tfrac q{200}-1$. Then
		$H-(A_1\cup A_2)$ has a matching of size $s+1$.
	\end{thm}
	
	\begin{proof}
		Put $a=\lfloor(q-1)/2\rfloor$. We use four properties
		of the family $Q(c,c')$, the first two of which refine
		Lemma~\ref{lem:blocks}.
		
		\medskip\noindent\textbf{Claim 1.} If $c^2\ge c'(q-c)$, then $g(c,c')\le cq+c'^2$; and
		if $q\ge6$ is even and $a=\tfrac q2-1$, then $\tau(q)^2=g(a,1)>aq+\tfrac12$.
		
		By Lemma~\ref{lem:blocks}(i), $g(c,c')$ is the larger root of
		$p(x)=x^2-(cq+c'q-cc')x+cc'(q-c)(q-c')$, and expanding gives
		$p(cq+\delta)=\delta^2+\delta(cq-c'(q-c))-cc'^2(q-c)$. For $X=cq+c'^2$ this is
		$c'^2(c'^2+c^2-c'(q-c))\ge0$, and $cq\ge c^2\ge c'(q-c)$ gives
		$2X\ge cq+c'(q-c)$, the sum of the roots, so $X$ is at least the larger root. For the
		second part, $p$ with $(c,c')=(a,1)$ is the polynomial defining $\tau(q)$, and with
		$q=2a+2$ one has $p(aq+\tfrac12)=\tfrac14-\tfrac12(3a+2)<0$.
		
		\medskip\noindent\textbf{Claim 2.} If $M$ is a graph with all edges in
		$(Y\setminus C)\times(Z\setminus C')$ and $G=Q(c,c')\cup M$, then
		$\rho(G)^2\le g(c,c')+\bigl(1+\tfrac{2c'q}{c^2}\bigr)e(M)$.
		
		In matrix products below, $M$ denotes its biadjacency matrix
		with rows indexed by $Y\setminus C$ and columns indexed by
		$Z\setminus C'$. Let $N$ and $N_0$ be the biadjacency matrices of $G$ and $Q(c,c')$, put
		$\lambda=\rho(G)^2$, and let $x\ge0$ be a unit vector with $N^{\top}Nx=\lambda x$.
		Write $X=\mathbf1^{\top}x$, let $x'$ be the restriction of $x$ to $Z\setminus C'$ and
		$p=\sum_{z\in C'}x_z$. For every $z$, $\lambda x_z=\sum_{y:N_{yz}=1}(Nx)_y$, and each
		$(Nx)_y$ is at most $X$ and equals $X$ for the $c$ all-ones rows, so
		$cX\le\lambda x_z\le qX$ and $\max_zx_z\le\tfrac qc\min_zx_z\le\sqrt q/c$. Since
		$(N_0x)_y=p$ and $(Nx)_y=p+(Mx')_y$ off $C$,
		$\lambda=|Nx|^2=|N_0x|^2+2p\,\mathbf1^{\top}Mx'+|Mx'|^2$, where
		$|N_0x|^2\le g(c,c')$, $|Mx'|^2\le e(M)$ by Lemma~\ref{lem:basic}(ii),
		$p\le c'\sqrt q/c$ and $\mathbf1^{\top}Mx'\le e(M)\sqrt q/c$.
		
		\medskip\noindent\textbf{Claim 3.} In the setting of Claim 2,
		if $c'=0$ and $G=Q(c,0)\cup M$ with
		$\lambda=\rho(G)^2>\mu:=\rho(M)^2$, then
		$(\lambda-cq)(\lambda-\mu)\le c\sum_yd_M(y)^2$.
		
		Here $N^{\top}N=cJ+W$ with $W=M^{\top}M$ positive semidefinite of largest eigenvalue
		$\mu$. If $x$ is a positive unit eigenvector for $\lambda$, then
		$(\lambda I-W)x=c(\mathbf1^{\top}x)\mathbf1$, so
		$1/c=\mathbf1^{\top}(\lambda I-W)^{-1}\mathbf1$; and
		$(\lambda I-W)^{-1}=\lambda^{-1}I+\lambda^{-1}W(\lambda I-W)^{-1}$ with
		$W(\lambda I-W)^{-1}\preceq(\lambda-\mu)^{-1}W$, whence
		$1/c\le q/\lambda+\mathbf1^{\top}W\mathbf1/(\lambda(\lambda-\mu))$ and
		$\mathbf1^{\top}W\mathbf1=\sum_yd_M(y)^2$.
		
		\medskip\noindent\textbf{Claim 4.} If $L_1,\dots,L_r$ with $r\ge2$ are bipartite
		graphs on a common bipartition, each with at least three edges, and every edge of
		$L_i$ meets every edge of $L_j$ for $i\ne j$, then the $L_i$ are stars with a common
		center.
		
		If $L_i$ had two disjoint edges $y_1z_1$, $y_2z_2$, every edge of $L_j$ would meet
		both and so be $y_1z_2$ or $y_2z_1$, leaving $L_j$ two edges. So $\nu(L_i)=1$ and
		$L_i$ is a star with a unique center $w_i$. If $w_1\ne w_2$, take leaves $u',u''$ of
		$L_2$ other than $w_1$; an edge $w_1u$ of $L_1$ meets $w_2u'$ and $w_2u''$, so
		$u\in\{w_2,u'\}\cap\{w_2,u''\}=\{w_2\}$ and $L_1$ has one edge.
		
		\medskip
		Let $B_i=V_i\setminus A_i$ and $F=H-(A_1\cup A_2)$, with classes $B_1,B_2,V_3$.
		Suppose $\nu(F)\le s$, put $\Theta=3(s+1)q$, and set
		$S=\{u\in V(F):d_F(u)>\Theta\}$, $S_i=S\cap V_i$, $F'=F-S$ and $t=s-|S|$. If $M_0$ is
		a matching of $F'$ and $U\subseteq S$ with $|M_0|+|U|\le s+1$, then $F$ has a matching
		of size $|M_0|+|U|$: cover the vertices $u\in U$ one at a time
		by edges containing $u$ and avoiding $V(M_0)$,
		$U\setminus\{u\}$ and all previously chosen edges,
		noting that at most $3(|M_0|+|U|-1)\le3s$
		vertices are to be avoided in the two classes not containing the current vertex and
		each lies in at most $q$ edges there, so fewer than $\Theta$ edges are excluded.
		Taking $M_0=\emptyset$ gives $|S|\le s$ and then $U=S$ gives $\nu(F')\le t$; hence
		\begin{equation}\label{eq:upperF}
			e(F')\le3t\Theta=9t(s+1)q .
		\end{equation}
		
		Put $A_i'=A_i\cup S_i$, $c_i=|A_i'|$ and $c_3=|S_3|$, so that
		\begin{equation}\label{eq:budget}
			c_1+c_2+c_3=q-1-t,\qquad 0.4q\le c_i\le0.6q\ (i=1,2),\qquad c_3\le s\le\tfrac q{200}.
		\end{equation}
		For $v\in V_1\setminus A_1'$, a pair $yz\in L_H(v)$ with $y\notin A_2'$ and
		$z\notin S_3$ comes from an edge of $F'$, so $L_H(v)\subseteq Q(c_2,c_3)\cup M_v$ with
		$e(M_v)=d_{F'}(v)$; by \eqref{eq:budget}, $2c_3q/c_2^2\le\tfrac1{16}$, so Claim 2
		gives
		\begin{equation}\label{eq:linkF}
			\tau(q)^2<\rho(L_H(v))^2\le g(c_2,c_3)+\tfrac{17}{16}d_{F'}(v),
		\end{equation}
		and symmetrically with $c_1$ in place of $c_2$. Also $c_i^2\ge0.16q^2\ge c_3(q-c_i)$,
		so $g(c_i,c_3)\le c_iq+c_3^2$ by Claim 1.
		
		Assume $c_2\le c_1$, put $k=t+c_3$ and $D=\tau(q)^2-g(c_2,c_3)$; by
		\eqref{eq:budget}, $c_2\le\lfloor(q-1-k)/2\rfloor$. We claim $D\ge0$, and
		$D\ge\tfrac13tq$ except when $q$ is even, $t=1$, $c_3=0$ and $c_1=c_2=a$. For odd $q$
		we have $\tau(q)^2=aq$ and $q-1=2a$, so $c_2\le a-\lceil k/2\rceil$ and
		$D\ge\tfrac12kq-c_3^2=\tfrac12tq+c_3(\tfrac q2-c_3)\ge\tfrac12tq$. For even $q$ we
		have $q-1=2a+1$ and $c_2\le a-\lfloor k/2\rfloor$. If $k\ge2$ then
		$\lfloor k/2\rfloor\ge k/3$ and $\tau(q)^2\ge aq$ give $D\ge\tfrac13tq$. If $k=0$
		then $t=0$ and $D\ge\tau(q)^2-aq>0$ by Claim 1. If $k=1$ and $c_3=1$ then $t=0$ and
		$D\ge\tau(q)^2-g(a,1)=0$ by Lemma~\ref{lem:blocks}(i) and Claim 1. If $k=1$ and
		$c_3=0$ then $t=1$ and $c_1+c_2=2a$; either $c_2\le a-1$ and $D\ge q$, or
		$c_1=c_2=a$.
		
		Suppose first $D\ge\tfrac13tq$. Since $|V_1\setminus A_1'|=q-c_1\ge0.4q$, if $t=0$
		then \eqref{eq:linkF} gives $d_{F'}(v)>0$ for all such $v$, contradicting
		$\nu(F')\le0$; and if $t\ge1$ then
		$d_{F'}(v)>\tfrac{16}{17}\cdot\tfrac13tq>\tfrac14tq$, and summing over
		$V_1\setminus A_1'$ gives $e(F')>\tfrac1{10}tq^2$, against \eqref{eq:upperF} and
		$s+1\le q/200$, which give $e(F')\le\tfrac9{200}tq^2$.
		
		There remains the case $q$ even, $t=1$, $c_3=0$, $c_1=c_2=a$, where the two sides are
		symmetric and $L_H(v)\subseteq Q(a,0)\cup M_v$ for every
		$v\in(V_1\setminus A_1')\cup(V_2\setminus A_2')$.
		For each such $v$, put $\lambda=\rho(Q(a,0)\cup M_v)^2$;
		spectral monotonicity gives $\lambda>\tau(q)^2$.
		First $d_{F'}(v)\ge3$: if $e(M_v)\le2$ then $\rho(M_v)^2\le2$ and $\sum_yd_{M_v}(y)^2\le4$, so Claim 3
		with $\lambda>aq$ gives $\lambda-aq\le4a/(aq-2)\le\tfrac12$, contradicting
		$\lambda>\tau(q)^2>aq+\tfrac12$. Since $\nu(F')\le1$, the graphs $M_v$ for
		$v\in V_1\setminus A_1'$ pairwise cross and there are at least two of them, so by
		Claim 4 they are stars with a common center $w$, and every edge of $F'$ contains $w$.
		If $w\in V_2$ the other vertices of $V_2\setminus A_2'$ have no $F'$-edge,
		contradicting $d_{F'}\ge3$; and if $w\in V_3$ then $L_H(v)\subseteq Q(a,1)$ for every
		$v\in V_1\setminus A_1'$, so $\rho(L_H(v))^2\le g(a,1)=\tau(q)^2$, again a
		contradiction.
	\end{proof}

	\begin{thm}\label{thm:extpart}
		Put $\varepsilon_0=10^{-11}$. There is $q_{\mathrm{ext}}$ such that the following
		holds for $q\ge q_{\mathrm{ext}}$. Let $H$ be $q$-balanced $3$-partite with
		$\sigma(H)>\tau(q)$, and suppose there are $A\subseteq X$ and $B\subseteq Y$ with
		$\bigl||A|-q/2\bigr|\le\varepsilon_0q$, $\bigl||B|-q/2\bigr|\le\varepsilon_0q$ and
		$e\bigl(H[X\setminus A,Y\setminus B,Z]\bigr)\le\varepsilon_0q^3$. Then $H$ has a
		perfect matching.
	\end{thm}
	
	\begin{proof}
		Fix $\delta=10^{-11}$, $\xi=10^{-3}$, $\gamma=10^{-4}$ and $\eta=10^{-6}$, and take
		$q_{\mathrm{ext}}\ge10^4$ large enough that $\tau(q)\ge(\tfrac1{\sqrt2}-\delta)q$,
		which is possible because $\tau(q)/q\to1/\sqrt2$. Then
		$\sigma(H)\ge(\tfrac1{\sqrt2}-\delta)q$ and \eqref{eq:degcross} holds, and
		Lemma~\ref{lem:crossmiss} gives $h:=|\mathcal B|\le12(\varepsilon_0+\delta)q/\xi<\eta q$.
		Write $d_{\mathrm{cr}}(v)$ for the number of crossing edges at $v$.
		
		\emph{Simultaneous adjustment.} Put
		$T_X=\{x\in X:d_{\mathrm{cr}}(x)<\gamma q^2\}$ and
		$T_Y=\{y\in Y:d_{\mathrm{cr}}(y)<\gamma q^2\}$, both computed from the original
		crossing degrees, and set $A^*=A\triangle T_X$ and $B^*=B\triangle T_Y$, the two
		changes being made simultaneously. A vertex of $X\cup Y$ lies in at least
		$(\tfrac12-\varepsilon_0)q^2$ legal crossing triples, so one outside $\mathcal B$ has
		more than $(\tfrac12-\varepsilon_0-\xi)q^2>\gamma q^2$ crossing edges; hence
		$T_X\cup T_Y\subseteq\mathcal B$ and $|T_X|+|T_Y|\le h$. Writing $a^*=|A^*|$ and
		$b^*=|B^*|$,
		\begin{equation}\label{eq:starsize}
			\bigl|a^*-\tfrac q2\bigr|,\ \bigl|b^*-\tfrac q2\bigr|\ \le\ \varepsilon_0q+h<2\eta q .
		\end{equation}
		Starred quantities refer to $A^*,B^*$. Every vertex has $d^*_{\mathrm{cr}}(v)\ge\tfrac\gamma2q^2$.
		Indeed, for $x\in X\setminus T_X$ the membership of $x$ is unchanged and flipping the
		vertices of $T_Y$ alters at most $|T_Y|q\le hq$ triples through $x$, so
		$d^*_{\mathrm{cr}}(x)\ge(\gamma-\eta)q^2$; for $x\in T_X$, flipping $x$ alone
		exchanges crossing and noncrossing edges at $x$, leaving $d_H(x)-d_{\mathrm{cr}}(x)$
		crossing edges, and the subsequent change of $B$ affects at most $hq$ triples, so
		$d^*_{\mathrm{cr}}(x)\ge(\tfrac12-2\delta-\gamma-\eta)q^2$; and the same two
		arguments apply in $Y$. For $z\in Z$, we use the biadjacency matrix of $L_H(z)$ on $X\times Y$. The noncrossing pairs for $A^*,B^*$ form
		$K_{a^*,b^*}\sqcup K_{q-a^*,q-b^*}$, whose matrix norm is
		$\max\{\sqrt{a^*b^*},\sqrt{(q-a^*)(q-b^*)}\}\le(\tfrac12+2\eta)q$ by
		\eqref{eq:starsize}; the noncrossing part of $L_H(z)$ has no larger norm by
		monotonicity, and its crossing part has Frobenius norm
		$\sqrt{d^*_{\mathrm{cr}}(z)}$, so
		$(\tfrac1{\sqrt2}-\delta)q\le(\tfrac12+2\eta)q+\sqrt{d^*_{\mathrm{cr}}(z)}$ and
		$d^*_{\mathrm{cr}}(z)>q^2/25$.
		
		\emph{The old good vertices stay good.} We do not recompute the exceptional set. A
		vertex $v\notin\mathcal B$ was not flipped, and the legal crossing triples through it
		change only when the vertex in the other of $X,Y$ was flipped, which affects at most
		$hq$ of them; for $z\in Z\setminus\mathcal B$ a crossing pair changes status only
		when one of its two ends was flipped, and there are at most $(|T_X|+|T_Y|)q\le hq$
		such pairs. Hence
		\begin{equation}\label{eq:stillgood}
			m_{A^*,B^*}(v)\le m_{A,B}(v)+hq\le(\xi+\eta)q^2<\frac{q^2}{144}
			\qquad(v\notin\mathcal B),
		\end{equation}
		and since every new type-$00$ triple that was not one before contains a vertex of
		$T_X\cup T_Y$, also $|E^*_{00}|\le|E_{00}|+hq^2<2\eta q^3$.
		
		\emph{Correcting the type count.} Put $d=q-a^*-b^*$, so that by \eqref{eq:starsize}
		$|d|<q/4000$ and $\tfrac{5q}{12}\le a^*,b^*\le\tfrac{7q}{12}$. We construct a
		matching $M_0$ with $|M_0|=|d|$ and $t^*_{11}(M_0)-t^*_{00}(M_0)=-d$. If $d>0$,
		Theorem~\ref{thm:defect} applied to $A^*,B^*$ gives a matching of size at least $d$
		in $H[X\setminus A^*,Y\setminus B^*,Z]$, all of whose edges have new type $00$; take
		$d$ of them. If $d=0$, take $M_0=\emptyset$. If $d<0$, split $L_H(z)$ as above,
		noting that the complete crossing pair graph for $A^*,B^*$ is
		$K_{a^*,q-b^*}\sqcup K_{q-a^*,b^*}$, again of norm at most $(\tfrac12+2\eta)q$, to
		get $d^*_{00}(z)+d^*_{11}(z)>q^2/25$ for every $z$; summing over $Z$ and using the
		bound on $|E^*_{00}|$ gives $|E^*_{11}|>q^3/50$. Choose $-d$ disjoint edges of new
		type $11$ greedily: after $j<-d$ choices at most $3jq^2\le12\eta q^3<q^3/50$ edges
		meet a chosen vertex, so another is available.
		
		\emph{Covering the exceptional vertices.} Starting from $M_0$, repeatedly take an
		uncovered $v\in\mathcal B$ and cover it by a crossing edge disjoint from the current
		matching; at most $h$ steps are needed, and throughout the matching has size
		$s\le|d|+h\le5\eta q$. An edge through an uncovered $v$ can meet the matching only in
		the other two classes, each containing $s$ covered vertices, and a covered vertex
		lies with $v$ in at most $q$ legal triples, so at most
		$2sq\le10\eta q^2<\tfrac\gamma2q^2$ crossing edges at $v$ are forbidden; one is
		therefore available. The resulting matching $M$ covers $\mathcal B$, has
		$|M|\le5\eta q<q/12$, and consists of $M_0$ together with edges of type $10$ or $01$,
		so $t^*_{11}(M)-t^*_{00}(M)=-d=a^*+b^*-q$. Every vertex outside $V(M)$ lies outside
		$\mathcal B$ and hence satisfies \eqref{eq:stillgood}, so Lemma~\ref{lem:complete}
		applied to $A^*,B^*$ extends $M$ to a perfect matching.
	\end{proof}

	\begin{proof}[Proof of Theorem~\ref{thm:nonpartite}]
		Let $\varepsilon_1$ and $\varepsilon_2$ be the constants
		of Lemmas~\ref{lem:ext1} and~\ref{lem:ext2}, and put
		$\varepsilon_*=\min\{\varepsilon_1,\varepsilon_2\}$ and
		$\varepsilon=\varepsilon_*/54$.
		Let $n_{\mathrm{ne}}$ be the order bound supplied by
		Theorem~\ref{thm:nonextnp} for this $\varepsilon$.
		Choose an integer $m_0$ large enough for
		Lemmas~\ref{lem:ext1} and~\ref{lem:ext2} and such that
		$3m_0\ge n_{\mathrm{ne}}$ and $m_0\ge9/\varepsilon_*$,
		and set $n_0=3m_0$.
		
		Let $H$ be a $3$-graph on $n=3m\ge n_0$ vertices with $\sigma(H)>2m-2$. If
		$|E(H)\setminus T_i(A)|\ge\varepsilon n^3$ for each $i\in\{1,2\}$ and every
		$A\subseteq V(H)$ of order $im$, then $H$ is $\varepsilon$-nonextremal and
		Theorem~\ref{thm:nonextnp} gives a perfect matching. Otherwise there are $i$ and $A$
		with $|A|=im$ and $|E(H)\setminus T_i(A)|<\varepsilon n^3$. Choose $a\in A$ and put
		$A_0=A\setminus\{a\}$; a triple in $T_i(A)\setminus T_i(A_0)$ contains $a$, so
		\[
		|E(H)\setminus T_i(A_0)|\ \le\ |E(H)\setminus T_i(A)|+\binom{n-1}2
		\ <\ 27\varepsilon m^3+\tfrac92m^2
		\ \le\ \tfrac{\varepsilon_*}2m^3+\tfrac{\varepsilon_*}2m^3=\varepsilon_*m^3 ,
		\]
		the last step using $m\ge9/\varepsilon_*$. Since $|A_0|=im-1$,
		Lemma~\ref{lem:ext1} applies when $i=1$ and Lemma~\ref{lem:ext2} when $i=2$, and in
		either case $H$ has a perfect matching.
		
		For sharpness, let $|V|=3m$ with $m\ge2$ and take $A\subseteq V$ with $|A|=m-1$ and
		$E(H)=T_1(A)$; every edge uses a vertex of $A$, so no matching has $m$ edges. For
		$b\notin A$ the link is $K_{m-1}\vee I_{2m}$, whose positive eigenvalue is the root
		of $\lambda^2-(m-2)\lambda-2m(m-1)$ by Lemma~\ref{lem:blocks}(iii); substituting
		$\lambda=2m-2$ gives zero, and the constant term is negative, so this is the
		positive root. For $a\in A$ the link is $K_{3m-1}$, of spectral radius $3m-2$. Hence
		$\sigma(H)=2m-2$. The second barrier attains the threshold as well: with
		$|A|=2m-1$ and $E(H)=T_2(A)$, a matching of $m$ edges would need $2m$ vertices of
		$A$; the link of $b\notin A$ is $K_{2m-1}$ together with isolated vertices, of
		spectral radius $2m-2$, and the link of $a\in A$ is $K_{2m-2}\vee I_{m+1}$, whose
		defining polynomial $\lambda^2-(2m-3)\lambda-(2m-2)(m+1)$ takes the value
		$-m(2m-2)<0$ at $\lambda=2m-2$ and so has its unique positive root above $2m-2$.
	\end{proof}
	
	\begin{proof}[Proof of Theorem~\ref{thm:partite}]
		Fix $\varepsilon_0=10^{-11}$, and let $\delta$ and $q_1$
		be given by Theorem~\ref{thm:nonextpart} for this
		$\varepsilon_0$.
		Increasing $q_1$ if necessary, we may assume that
		$\tau(q)\ge(1/\sqrt2-\delta)q$ for every $q\ge q_1$,
		since $\tau(q)/q\to1/\sqrt2$. Let $q_0=\max\{q_1,q_{\mathrm{ext}}\}$ and let $H$ be
		$q$-balanced $3$-partite with $q\ge q_0$ and $\sigma(H)>\tau(q)$. If $H$ is
		$\varepsilon_0$-nonextremal, Theorem~\ref{thm:nonextpart} gives a perfect matching.
		Otherwise there are distinct classes and sets $A,B$ in them of size within
		$\varepsilon_0q$ of $q/2$ such that at most $\varepsilon_0q^3$ edges avoid
		$A\cup B$, and Theorem~\ref{thm:extpart} applies after relabelling. The
		constructions of Section~\ref{sec:intro} show that the threshold is best possible.
	\end{proof}
	
	\section{Concluding remarks}\label{sec:conc}
	
	Both theorems are proved for large order, and the conjectures of~\cite{LLYZ}
	and~\cite{LY} remain open for the remaining small cases. We have not tried to
	optimize the constants, and the order thresholds our argument produces are very
	large; in the partite extremal analysis, we fix a positive
	spectral tolerance $\delta$ and then take $q$ large enough
	that $1/\sqrt2-\tau(q)/q\le\delta$.
	It would be interesting to know whether the exact thresholds hold for every $n$ and
	$q$ above a small explicit bound.
	
	The conjecture of Lin, Lu, Yuan and Zhao is more general than
	Theorem~\ref{thm:nonpartite}: it predicts, for each $k$, the least $\sigma$ forcing
	a matching of size $k$, the perfect matching endpoint being the case $3k=n$. Our
	argument uses that endpoint in an essential way, through the balance conditions used in the extremal
	matching constructions, and
	we do not see how to run it for smaller $k$. It would also be natural to seek a matching-number
	analogue of the partite threshold.
	
	Finally, the comparison between the two settings seems worth pursuing. The partite
	barrier is a single family with two barrier sets of sizes close to $q/2$, and this rigidity
	is what allows the whole stability analysis to be carried by one inequality for
	bipartite graphs, Lemma~\ref{lem:bipineq}. The corresponding non-partite statement
	has to allow the two families $T_1(A)$ and $T_2(A)$ simultaneously, and its proof is
	correspondingly longer. It would be interesting to investigate whether a similar
	contrast holds between $k$-partite $k$-graphs and general
	$k$-graphs for $k\ge4$.
	
	\section*{Declaration of generative AI and AI-assisted technologies in the manuscript preparation process}
	
	During the preparation of this work, the authors used Claude (Anthropic) for language refinement, technical editing, and computational verification of examples. The authors reviewed and edited the output as needed and take full responsibility for the content of the published article.


\begin{thebibliography}{99}
		
		\bibitem{BH}
		A. E. Brouwer and W. H. Haemers, Spectra of graphs, Springer, New York, 2012.
		
		\bibitem{FR}
		P. Frankl and V. R\"odl, Near perfect coverings in graphs and hypergraphs,
		European J. Combin. 6 (1985) 317--326.
		
		\bibitem{GR}
		C. Godsil and G. Royle, Algebraic graph theory, Springer, New York, 2001.
		
		\bibitem{LLYZ}
		H. Lin, H. Lu, F. Yuan and X. Zhao, A local spectral condition for perfect
		matchings in $3$-graphs, preprint, arXiv:2604.13726, 2026.
		
		\bibitem{LY}
		H. Lu and F. Yuan, A spectral condition for perfect matchings in $3$-partite
		$3$-graphs, preprint, arXiv:2606.15771, 2026.
		
		\bibitem{Stanley}
		R. P. Stanley, A bound on the spectral radius of graphs with $e$ edges,
		Linear Algebra Appl. 87 (1987) 267--269.
		
		\bibitem{West}
		D. B. West, Introduction to graph theory, 2nd edition, Prentice Hall, Upper
		Saddle River, 2001.
		
	\end{thebibliography}
\end{document}